\documentclass[reqno]{amsart}
\usepackage{amsthm,amsmath,amssymb,amsthm}
\usepackage[top=1.5cm, bottom=1.5cm, left=2.5cm, right=2.5cm, includefoot]{geometry}
\usepackage[all]{xy}
\usepackage{color}

\theoremstyle{definition}
\newtheorem{theorem}{Theorem}[section]
\newtheorem{proposition}{Proposition}[section]
\newtheorem{lemma}{Lemma}[section]
\newtheorem{corollary}{Corollary}[section]

\newtheorem{example}{Example}[section]
\newtheorem{definition}{Definition}[section]
\newtheorem{remark}{Remark}[section]
\newtheorem*{theoremA}{Theorem}
\numberwithin{equation}{section}

\begin{document}
\title[Quantized Vershik--Kerov Theory and Quantized Central measures]{Quantized Vershik--Kerov Theory and Quantized Central Measures on Branching Graphs}
\author[R. Sato]{Ryosuke SATO}
\address{Graduate School of Mathematics, Kyushu University, Fukuoka 819-0395, Japan}
\email{ma217052@math.kyushu-u.ac.jp}
\maketitle

\begin{abstract}
We propose a natural quantized character theory for inductive systems of compact quantum groups based on KMS states on AF-algebras following Stratila--Voiculescu's work \cite{StraVoic:book} (or \cite{EnomotoIzumi}), and give its serious investigation when the system consists of quantum unitary groups $U_q(N)$ with $q\in(0,1)$. The key features of this work are: The ``quantized trace" of a unitary representation of a compact quantum group can be understood as a quantized character associated with the unitary representation and its normalized one is captured as a KMS state with respect to a certain one-parameter automorphism group related to the so-called scaling group. In this paper we provide a Vershik--Kerov type approximation theorem for extremal quantized characters (called the ergodic method) and also compare our quantized character theory for the inductive system of $U_q(N)$ with Gorin's theory on $q$-Gelfand--Tsetlin graphs \cite{Gorin:qcentralmeasure}.
\end{abstract}

\allowdisplaybreaks{ 

\section{Introduction}
\subsection{Preface}
Voiculescu \cite{Voiculescu} initiated the study of extremal characters of the infinite-dimensional unitary group $U(\infty)=\varinjlim U(N)$, and then Vershik and Kerov \cite{VershikKerov2} proved based upon their so-called ergodic method, among other things, that Voiculescu's list of extremal characters is indeed complete (see also an independent work due to Boyer \cite{Boyer}). The (extremal) characters precisely correspond to a certain class of (extremal) tracial states on a certain AF-algebra, called the Stratila--Voiculescu AF-algebra (following Enomoto--Izumi \cite{EnomotoIzumi} for this name), canonically constructed from the inductive system of $U(N)$ (see \cite{StraVoic:book}). Thus the characters of $U(\infty)$ can be investigated within the framework of operator algebras completely. This fact is analogous to the well-known fact that any normalized character on a finite group $\Gamma$ can be captured as a normalized trace (tracial state) on the group algebra $\mathbb{C}[\Gamma]$, and hence the characters of $\Gamma$ can be investigated by using the group algebra $\mathbb{C}[\Gamma]$ rather than $\Gamma$ itself. In this context, the Stratila--Voiculescu AF-algebra associated with the inductive system of $U(N)$ should be regarded as a natural topological group algebra of $U(\infty)$. Moreover, Vershik and Kerov discovered that the above-mentioned class of tracial states (which precisely corresponds to the family of characters of $U(\infty)$) on the Stratila--Voiculescu AF-algebra precisely corresponds to the family of \emph{central probability measures} on the paths on the Gelfand--Tsetlin graph. By these correspondences among characters, tracial states and central probability measures, the characters of $U(\infty)$ can be analyzed in the framework of probability theory like Vershik--Kerov's ergodic method (\cite{VershikKerov2}) and Borodin--Olshanski's harmonic analysis (\cite{Olshanski}, \cite{BorodinOlshanski3} etc.). These works due to Vershik and Kerov, and Borodin and Olshanski form an important part of the asymptotic representation theory of $U(\infty)$.

On the other hand, Gorin \cite{Gorin:qcentralmeasure} introduced the concept of $q$-central probability measures on the paths on the Gelfand--Tsetlin graph and obtained the complete classification for extremal points of the simplex consisting of $q$-central probability measures. The concept of $q$-central probability measures comes from quantum traces (of irreducible representations) of the quantized universal enveloping algebra $\mathcal{U}_\epsilon(\mathfrak{gl}(N))$ as was remarked by Gorin himself in \cite[Remark 4]{Gorin:qcentralmeasure}, and it made him conjecture that the $q$-central probability measures are precisely associated with the representation theory of a certain quantum group. In this paper, we will give an answer to Gorin's conjecture along the line of the asymptotic representation theory of $U(\infty)$, especially Vershik and Kerov's work mentioned above. Namely, we will introduce the notion of quantized characters of inductive systems of compact quantum groups such as the inductive system of $U_q(N)$ and investigate them within the framework of operator algebras. Then we will prove that the quantized characters correspond to Gorin's $q$-central probability measures in an explicit way. In particular, we will give a representation-theoretic interpretation of the weights of finite paths on the Gelfand--Tsetlin graph introduced by Gorin \cite{Gorin:qcentralmeasure} to define the $q$-centrality; these weights come from weights of irreducible representations of $U_q(N)$. However there are two apparent difficulties in the attempt to give an answer to Gorin's conjecture. The first one is a formulation of the ``infinite-dimensional quantum unitary group $U_q(\infty)=\varinjlim U_q(N)$'' itself in the spirit of Woronowicz. The second is what should be an appropriate notion of ``quantized characters''. These two problems are really non-trivial at all, and indeed nobody has worked on this natural attempt. For the first difficulty, the quantum unitary group $U_q(N)$ in the sense of Woronowicz is defined as a certain pair of unital $C^*$-algebra and coproduct. Note that this unital $C^*$-algebra should be regarded as the ``continuous function algebra'' over $U_q(N)$. Thus, the inductive system of $U_q(N)$ is understood as the projective system of the $C^*$-algebras of $U_q(N)$, and the infinite-dimensional quantum unitary group $U_q(\infty)$ has already been constructed in terms of $\sigma$-$C^*$-algebras (see \cite{MahantaMathai}). However, as explained in the first paragraph, we need a suitable group algebra of $U_q(\infty)$, which should be like the Stratila--Voiculescu AF-algebra of $U(\infty)$ for our purpose. Hence we will construct the ``Stratila--Voiculescu AF-algebra'' directly from the inductive system of $U_q(N)$ rather than $U_q(\infty)$, and regard it as a suitable group algebra of $U_q(\infty)$. We leave it as a future work to establish the relationship (like the Hopf-Algebra duality) between the $U_q(\infty)$ itself and the Stratila--Voiculescu AF-algebra that we will construct from the inductive system of $U_q(N)$.

\subsection{Concept of quantized characters}
We will first discuss what should be an appropriate notion of ``quantized characters'' for compact quantum groups rather than inductive systems of them. Since characters of groups can be understood as traces on their group algebras, we should define a quantized character of a compact quantum group to be a certain linear functional on the group algebra that is obtained via the quantum group duality. However, it is known that $U(N)$ and $U_q(N)$ for example have the same ``representation theory''; namely, their group algebras must have the same structure. Thus, we cannot distinguish $U(N)$ and $U_q(N)$ in terms of their group algebras, and hence we need to find an additional notion to define quantized characters as linear functionals on group algebras. This is the heart of the problem here. Our idea is to focus on the quantum dimensions of unitary representations and the so-called scaling groups of compact quantum groups. The idea to focus on quantum dimensions is essentially the same idea as Gorin's one for his definition of $q$-central probability measures. On the other hand, the idea to use scaling groups is completely new and essential in this paper, though it is quite natural in view of Woronowicz's theory of quantum groups.

Let $G$ be a compact quantum group and $U$ a (finite-dimensional) irreducible unitary representation of $G$ (see Section \ref{CQGs} for the definition). As opposed to ordinary groups, the double contragredient representation $U^{cc}$ is not unitary in general, that is, two representations $U$ and $U^{cc}$ are not unitarily equivalent. However, there exists a unique positive invertible intertwiner $F_U$ from $U$ to $U^{cc}$ such that $\mathrm{Tr}(F_U)=\mathrm{Tr}(F_U^{-1})$, which is the quantum dimension of $U$. Then we can define the \emph{quantized character} assigned to $U$ to be the linear functional $\mathrm{Tr}(F_U\,\cdot\,)$, which can be regarded as a linear functional on the group algebra of $G$ precisely formulated in this paper. In order to formulate general quantized characters, we use the scaling group of $G$. The scaling group of $G$ plays a special role in Woronowicz's theory of quantum groups. Moreover, the scaling group of $G$ is determined by the intertwiners $F_U$ over all the irreducible unitary representations, which turns to be the trivial one at least for ordinary groups. Observe that the ``dual'' of the scaling group of $G$ can be considered as a one-parameter automorphism group of the group algebra of $G$ and characterizes the normalized quantized character $\mathrm{Tr}(F_U\,\cdot\,)/\mathrm{Tr}(F_U)$ assigned to an irreducible unitary representation $U$ as an extremal KMS state for this dual of the scaling group of $G$. In this way, we can define quantized characters of $G$ in terms of the dual of the scaling group of $G$. For our purpose, we will generalize this idea to the setting of inductive systems of compact quantum groups based on their Stratila--Voiculescu AF-algebras.

\subsection{Main results and connections with other works}
In this paper, we will consider rather general inductive systems of compact quantum groups, but our main example is the inductive system of $U_q(N)$. For the inductive system of $U_q(N)$, a part of what we will obtain in this paper can be summarized as follows.

\begin{theoremA}\label{TheoremA}
There exists an explicit affine homeomorphism between the simplex of quantized characters of the inductive system of $U_q(N)$ and the simplex of Gorin's $q^2$-central probability measures on the paths on the Gelfand --Tsetlin graph.
\end{theoremA}

This theorem gives a representation-theoretic interpretation of Gorin's work on the $q$-Gelfand--Tsetlin graph (modulo the change of parameter $q\mapsto q^2$). See Theorem \ref{main} for more details. Moreover, we will give a representation-theoretic interpretation of generating functions of probability measures on the sets of signatures introduced by Gorin \cite{Gorin:qcentralmeasure} to analyze $q$-central probability measures. Namely, Theorem \ref{Theorem:rep_qcharacters} asserts that these generating functions coincide with the restrictions of quantized characters to the maximal tori of $U_q(N)$. In this way, we will find a quantum group analogue of the so-called Voiculescu functions (\cite{Voiculescu}). We will also establish the dynamical interpretation of $q$-central probability measures along the line of Vershik--Kerov's idea and construct the GNS representation of the Stratila--Voiculescu AF algebra associated with a given quantized character of the inductive system of $U_q(N)$ by using the corresponding $q$-central probability measure and its dynamical interpretation. Concerning the harmonic analysis for the inductive system of $U_q(N)$, we would like to mention that the above theorem and Gorin's work \cite{Gorin:qcentralmeasure} together completely classify all the extremal quantized characters of the inductive system of $U_q(N)$, see Appendix A. On the other hand, the decomposition problem of a given quantized character into extremal ones remains as a future work, and it is also a future interesting problem to find any possible relationship between the present work and the quantized harmonic analysis due to Gorin and Olshanski \cite{GorinOlshanski}.

\subsection{Organization of the paper}
This paper consists of two parts. The first one (Section 2) is a general theory of quantized characters. Firstly, we review some basic definitions and some properties of compact quantum groups and fix some notations used throughout this paper. In Section 2.2, we deal with the duals of compact quantum groups in order to formulate their group ($C^*$ or von Neumann) algebras. In Section 2.3, we discuss compact quantum groups and their quantum subgroups (remark that pairs of compact quantum groups and quantum subgroups are basic building blocks of inductive systems of compact quantum groups). The purpose of this section is an explicit intertwining formula (see Proposition \ref{P4.2} and Equation \eqref{Eq:Density}) of the duals of scaling groups of a compact quantum group and its quantum subgroup. In Section 2.4, we deal with inductive systems of compact quantum groups and construct their Stratila--Voiculescu AF algebras (following Enomoto--Izumi \cite{EnomotoIzumi}, but in a bit different way from theirs). Moreover, we introduce a canonical flow on the Stratila--Voiculescu AF-algebra associated with a given inductive system of compact quantum groups, which is essentially obtained as the inductive limit of the duals of scaling groups. Then a quantized character of the given inductive system of compact quantum groups is defined in terms of this flow, or more precisely, as a certain KMS state with respect to this flow. Section 2.5 devotes to the ergodic method for quantized characters. Moreover, we investigate a natural dynamical system on the paths on the branching graph of a given inductive system of compact quantum groups, and obtain that the extremity for $q$-central probability measures coincides with the ergodicity for $q$-central probability measures when the given compact quantum groups are $U_q(N)$ (actually a more general assertion is given). In the second part of this paper (Section 3) we study the inductive system of $U_q(N)$. We calculate some representation-theoretic quantities (for instance, Equation \eqref{density} and \eqref{Def:weight}) by using the Gelfand--Tsetlin basis. Moreover, we obtain the main theorem (see Theorem \ref{main}) by using these quantities. Finally, we give a representation-theoretic interpretation of Gorin's generating functions of probability measures on the set of signatures. 

There are two appendices. In Appendix A, we briefly explain Gorin's boundary theorem (i.e., the complete parametrization of extremal quantized characters of the inductive system of $U_q(N)$, see \cite{Gorin:qcentralmeasure}) as a corollary of the ergodic method in Section 2. In Appendix B, we briefly touch Cuenca's recent work \cite{Cuenca} on $(q,t)$-central probability measures on the paths of the Gelfand--Tsetlin graph from our point of view.

\section{General theory}
\subsection{Compact quantum groups}\label{CQGs} 

We recall some basic definitions and some properties on compact quantum groups to fix some notations used throughout this paper. A typical example is the quantum unitary group $U_q(N)$, which will be discussed in Section 3. 

\medskip
Let $G=(A,\delta)$ be a pair of a unital $C^*$-algebra and a unital $*$-homomorphism $\delta\colon A\to A\otimes A$ (called the comultiplication), where $\otimes$ denotes the minimal (or spatial) tensor product. The pair $G=(A,\delta)$ is called a \emph{$C^*$-algebraic compact quantum group} (CQG) if it satisfies the following two conditions:
\begin{itemize}
\item (coassociativity) $(\mathrm{id}\otimes\delta)\delta=(\delta\otimes\mathrm{id})\delta$ as homomorphisms from $A$ to $A\otimes A\otimes A$,
\item (cancellation property) the spaces $\delta(A)(A\otimes 1_A)$ and $\delta(A)(1_A\otimes A)$ is dense in $A\otimes A$.
\end{itemize}

For two given $f_1,f_2\in A^*$ we define $f_1*f_2:=(f_1\otimes f_2)\circ\delta\in A^*.$ It is well known (see e.g.\, \cite[Theorem 1.2.1]{Neshveyev}) that any CQG $G=(A,\delta)$ has the so-called \emph{Haar state} $h:A\to\mathbb{C}$, which enjoys that $f*h=h*f=f(1_A)h$ for any linear functional $f\in A^*$, or equivalently $(\mathrm{id}\otimes h)\delta(a)=(h\otimes\mathrm{id})\delta(a)=h(a)1_A$ for any $a\in A$.

\medskip
Let $G=(A,\delta)$ be a CQG and $\mathcal{V}$ be a finite-dimensional vector space. An invertible element $U\in B(\mathcal{V})\otimes A$ is called a \emph{representation} of $G$ if it satisfies $(\mathrm{id}\otimes\delta)(U)=U_{12}U_{13}$ in $B(\mathcal{V})\otimes A\otimes A,$ where $U_{12}$, $U_{13}$, etc are leg numbering notations, see \cite[Section 1.3]{Neshveyev}. Let $e_{ij}$ be a matrix unit system of $B(\mathcal{V}).$ An invertible element $U=\sum_{i,j=1}^{\dim\mathcal{V}}e_{ij}\otimes u_{ij}\in B(\mathcal{V})\otimes A$ is a representation of $G$ if and only if $\delta(u_{ij})=\sum_{k=1}^{\dim\mathcal{V}}u_{ik}\otimes u_{kj}$ holds for any $i,j=1,\dots,\dim\mathcal{V}$. The element $u_{ij}$ is called the \emph{matrix coefficient} of $U$ with respect to the matrix unit system $e_{ij}$ (or the basis of $\mathcal{V}$). The dimension of $\mathcal{V}$ is denoted by $\dim(U)$ (or $\dim(\pi)$ when $U = U_\pi$) and called the \emph{dimension} of the representation $U$. If $\mathcal{V}$ is a Hilbert space, that is, $\mathcal{V}$ is equipped with an inner product and $U$ is a unitary, then $U$ is called a \emph{unitary representation} of $G$. For any two finite-dimensional representations $U,V$ on vector spaces $\mathcal{V}_U,\mathcal{V}_V$, respectively, a linear map $T\colon\mathcal{V}_U\to\mathcal{V}_V$ is an \emph{intertwiner} from $U$ to $V$ if $(T\otimes\mathrm{id})U=V(T\otimes\mathrm{id})$ holds. We denote by $\mathrm{Mor}(U,V)$ the intertwiners from $U$ to $V$. If there exists a bijective intertwiner from $U$ to $V$ (or from $V$ to $U$), then $U$ and $V$ are said to be \emph{equivalent}, moreover if these representations are unitary and there exists a unitary intertwiner then they are said to be \emph{unitarly equivalent}, we write $U\sim V$ in the case. We denote by $\widehat{G}$ all the equivalence classes of irreducible unitary representations, and call $\widehat{G}$ the \emph{unitary dual} of $G$. Let $U$ be a finite-dimensional representation on a Hilbert space $\mathcal{H}$. Let $J\colon\mathcal{H}\to\mathcal{H}^*$ be a conjugate linear map sending a basis to its dual basis and $j\colon B(\mathcal{H})\ni a\mapsto Ja^*J^{-1}\in B(\mathcal{H}^*)$. Then a new representation $U^c:=(j\otimes\mathrm{id})U^{-1}\in B(\mathcal{H}^*)\otimes A$ is called the \emph{contragredient representation} of $U$. It is well known that, for any finite-dimensional unitary representation $U$, there exists a unique positive invertible intertwiner $F_U\in\mathrm{Mor}(U,U^{cc})$ such that $\mathrm{Tr}(F_U\,\cdot\,)=\mathrm{Tr}(F_U^{-1}\,\cdot\,)$ on $\mathrm{End}(U):=\mathrm{Mor}(U,U)$. The trace $\mathrm{Tr}(F_U)$ is called the \emph{quantum dimension} of the representation $U$, denoted by $\dim_q(U)$. In this paper, we call the matrix $F_U$ the \emph{density matrix} of the representation $U$. In what follows, we write $\mathcal{H}_\pi := \mathcal{H}_{U_\pi}$, $\dim(\pi):=\dim(\mathcal{H}_\pi),$ $F_\pi:= F_{U_\pi}$ and $\dim_q(\pi) := \dim_q(U_\pi)$ when given representations $U_\pi$ have suffix $\pi$.  For a given unitary representation $U$, if $\mathrm{End}(U)$ is $1$-dimensional, then $U$ is said to be \emph{irreducible}. It is well known that every finite dimensional representation of a CQG becomes a direct sum of irreducible ones, see \cite[Theorem 1.3.7]{Neshveyev}. 

\medskip
Let $\mathcal{A}\subset A$ be the subspace generated by matrix coefficients of finite-dimensional representations. By the definition of tensor product representations and contragredient representations, the subspace $\mathcal{A}$ becomes a $*$-subalgebra of $A$, see \cite[section 1.6]{Neshveyev}. Conversely, we always assume that $A$ is the universal $C^*$-algebra generated by $\mathcal{A}$. For any $z\in\mathbb{C}$, we have a linear functional $f_z^G\colon\mathcal{A}\to\mathbb{C}$ determined by $(\mathrm{id}\otimes f_z^G)(U)=F_U^z$ for any finite-dimensional unitary representation $U$. See \cite[section 1.7]{Neshveyev}. These functionals $\{f^G_z\}_{z\in\mathbb{C}}$ are called the \emph{Woronowicz characters} of the CQG $G$. By definition, we have $F_U^z=[f_z^G(u_{ij})]_{ij=1}^{\dim(U)}$, where $u_{ij}$ are matrix coefficients of the representation $U$. For any $a\in A$ and $f,g\in \mathcal{A}^*$, we define $f*a:=(\mathrm{id}\otimes f)\delta(a)$, $a*f=(f\otimes\mathrm{id})\delta(a)$ and $f*a*g:=f*(a*g)=(f*a)*g$. The Woronowicz characters $\{f_z^G\}_z$ induce two actions $\sigma^G,\tau^G : \mathbb{C} \curvearrowright \mathcal{A}$ of the whole complex field $\mathbb{C}$ on $\mathcal{A}$ defined by $\sigma_z^G(a):=f_{\sqrt{-1}z}^G*a*f_{\sqrt{-1}z}^G$ and $\tau_z^G(a):=f_{-\sqrt{-1}z}^G*a*f_{\sqrt{-1}z}^G$  for any $a\in\mathcal{A}$, and they are called the \emph{modular group} (or \emph{modular action}) and the \emph{scaling group} (or \emph{scaling action}), respectively. Note that 
\begin{equation}\label{Eq:Modular_Scaling_Action}
(\mathrm{id}\otimes\sigma_z^G)(U)=(F_U^{\sqrt{-1}z}\otimes 1_A)U(F_U^{\sqrt{-1}z}\otimes 1_A), \quad
(\mathrm{id}\otimes\tau_z^G)(U)=(F_U^{\sqrt{-1}z}\otimes 1_A)U(F_U^{-\sqrt{-1}z}\otimes 1_A)
\end{equation} 
for any finite-dimensional unitary representation $U$. We remark that the Haar state $h$ on $\mathcal{A}$ is $\sigma^G_z$-invariant for all $z\in\mathbb{C}$ and satisfies $h(ab)=h(b\sigma_{-\sqrt{-1}}(a))$ for all $a,b\in\mathcal{A}$, see \cite[Theorem 1.7.3]{Neshveyev}. This fact is the reason why $\sigma_z$ is called the modular group.

\medskip
In closing of this section we recall the notion of quantum subgroups. Let $G=(A,\delta_G)$ and $H=(B,\delta_H)$ be CQGs. The CQG $H$ is a \emph{quantum subgroup} of the CQG $G$ if there exists a surjective $*$-homomorphism $\theta\colon A\to B$ which satisfies $\delta_H\theta=(\theta\otimes\theta)\delta_G\colon A\to B\otimes B$. Since $\theta$ is a $*$-homomorphism, we have the following: 
For any unitary representation $U$ of $G$, it is easy to see that $(\mathrm{id}\otimes\theta)(U)$ is a unitary representation of the quantum subgroup $H$, and this is called the restriction of $U$ to $H$. When the $U$ has a suffix $\pi$, that is, $U = U_\pi$, we denote the restriction of $U_\pi$ to $H$ by $\pi|_H$ and write $U_{\pi|_H} = (\mathrm{id}\otimes\theta)(U_\pi)$.

\subsection{Duals of compact quantum groups}
Let $G=(A,\delta)$ be a CQG and $\mathcal{A}$ the $*$-subalgebra generated by the matrix coefficients of all finite-dimensional representations of $G$. It is known that $\mathcal{A}$ becomes a Hopf $*$-algebra with comultiplication $\delta|_\mathcal{A}$, counit $\epsilon$ and antipode $S$ determined by $\epsilon(u_{ij})=\delta_{i,j}$, $S(u_{ij})=u_{ji}^*$ for matrix coefficients $u_{ij}$ of every finite-dimensional representation $U$. See \cite[Theorem 1.6.4]{Neshveyev}. 

\medskip
The dual $\mathcal{U}(G) := \mathcal{A}^*$ as a linear space becomes a $*$-algebra with multiplication \[\mathcal{U}(G)\times\mathcal{U}(G) \ni (f_1,f_2) \mapsto f_1 * f_2 := (f_1\otimes f_2)\circ\delta \in \mathcal{U}(G)\] and involution $\mathcal{U}(G) \ni f \mapsto f^* \in \mathcal{U}(G)$ defined to be $f^*(a) := \overline{f(S(a^*))}$, $a \in \mathcal{A}$. In what follows, we choose and fix a complete family, say $\{U_\pi\}_{\pi \in \widehat{G}}$, of representatives of members of $\widehat{G}$. It is known, see e.g.\ \cite[Section 1.6]{Neshveyev}, that for each $\pi\in\widehat{G}$ the mapping $f\in\mathcal{U}(G)\mapsto U_\pi(f):=(\mathrm{id}\otimes f)(U_\pi)\in B(\mathcal{H_\pi})$ defines a surjective $*$-homomorphism and moreover that \[U_{\widehat{G}}\colon \mathcal{U}(G)\ni f\mapsto U_{\widehat{G}}(f):=(U_\pi(f))_{\pi\in\widehat{G}}\in\prod_{\pi\in\widehat{G}}B(\mathcal{H}_\pi)\] becomes a bijective $*$-homomorphism. The surjectivity and the bijectivity of the mappings $f\mapsto U_\pi(f)$ and $f\mapsto U_{\widehat{G}(f)}$, respectively, are bit non-trivial and follow from the following consideration: For each $\pi\in\widehat{G},$ we set \[a_{ij}(\pi):=\dim_q(\pi)\sum_{p=1}^{\dim(\pi)}f_{-1}(u_{jp}(\pi))u_{ip}(\pi)^*\] with $U_\pi=\sum_{i,j=1}^{\dim(\pi)}e_{ij}(\pi)\otimes u_{ij}(\pi)$. Then we have \[U_\pi(a_{ij}(\pi)h)=e_{ij}(\pi)\] by the orthogonality relations for matrix coefficients, where $[a_{ij}(\pi)h](\,\cdot\,):=h(\,\cdot\,\ a_{ij}(\pi))$. See \cite[Theorem 1.4.3]{Neshveyev}. It follows that $f\mapsto U_\pi(f)$ and $f\mapsto U_{\widehat{G}(f)}$ are surjective and moreover that 
\begin{equation}\label{Eq3.1}
U_{\widehat{G}}^{-1}\Bigg(\Big(\sum_{i,j=1}^{\dim(\pi)}\alpha_{ij}(\pi)e_{ij}(\pi)\Big)_{\pi\in\widehat{G}}\Bigg)=\sum_{\pi\in\widehat{G}}\sum_{i,j=1}^{\dim(\pi)}\alpha_{ij}(\pi)a_{ij}(\pi)h,
\end{equation}
whose right-hand side involves an infinite sum over $\pi\in\widehat{G},$ but it is indeed a well-defined linear functional on $\mathcal{A},$ because 
\[
[a_{ij}(\pi)h](u_{kl}(\rho))=h(u_{kl}(\rho)a_{ij}(\pi))=\delta_{\pi,\rho}\delta_{i,k}\delta_{j,l},
\]
that is, the $a_{ij}(\pi)h$ form a dual basis of the $u_{ij}(\pi)$. 

\medskip
Here is a simple (probably well-known) lemma, which immediately follows from \eqref{Eq:Modular_Scaling_Action}.

\begin{lemma}\label{lemma:dualscalingaction_functional}
For every $\pi\in\widehat{G}$ we have \[U_\pi(f\circ\tau^G_t)=F^{\sqrt{-1}t}_\pi U_\pi(f)F^{-\sqrt{-1}t}_\pi,\quad f\in\mathcal{U}(G), \ t\in\mathbb{R}.\] Therefore, the dual scaling group $\widehat{\tau}^G:\mathbb{R}\curvearrowright\mathcal{U}(G)$ defined by $\widehat{\tau}_t^G(f):=f\circ\tau_t^G$ for every $f\in\mathcal{U}(G)$ and $t\in\mathbb{R}$ enjoys the formula \[U_{\widehat{G}}(\widehat{\tau}_t^G(f))=\Big(F^{\sqrt{-1}t}_\pi U_\pi(f)F^{-\sqrt{-1}t}_\pi\Big)_{\pi\in\widehat{G},}\quad f\in\mathcal{U}(G),\ t\in\mathbb{R},\] and  $\hat{\tau}^G_t(f) = f^G_{\sqrt{-1}t}*f*f^G_{-\sqrt{-1}t}$ holds for every $f \in \mathcal{U}(G)$ and $t \in \mathbb{R}$. 
\end{lemma}

\bigskip
There are three canonical $*$-subalgebras of $\mathcal{U}(G)$ or $\prod_{\pi\in\widehat{G}} B(\mathcal{H}_\pi)$. The collection of all elements $(x_\pi)_{\pi \in \widehat{G}}$ in $\prod_{\pi\in\widehat{G}} B(\mathcal{H}_\pi)$ such that $\sup_{\pi \in \widehat{G}} \Vert x_\pi\Vert_\infty < +\infty$ becomes a unital $*$-subalgebra and is denoted by $\bigoplus_{\pi\in\widehat{G}} B(\mathcal{H}_\pi)$. Remark that this unital $*$-subalgebra becomes a von Neumann algebra in an obvious way and its image, denoted by $W^*(G)$, in $\mathcal{U}(G)$ by the mapping $f\mapsto U_{\widehat{G}}(f)$ is exactly the unital $*$-subalgebra consisting of all $f \in \mathcal{U}(G)$ with $\sup_{\pi \in \widehat{G}}\Vert U_\pi(f)\Vert_\infty < +\infty$. We call $W^*(G)$ the \emph{group von Neumann algebra} associated with $G$ in what follows. The algebraic direct sum $\bigodot_{\pi\in\widehat{G}}B(\mathcal{H}_\pi)$, which sits inside $\bigoplus_{\pi\in\widehat{G}} B(\mathcal{H}_\pi)$, is also a non-unital $*$-subalgebra, and its image, denoted by $\mathbb{C}[G]$, in $W^*(G)$ by the mapping $f\mapsto U_{\widehat{G}}(f)$ is exactly the non-unital $*$-subalgebra consisting of all linear combinations of the $a_{ij}(\pi)h$. This should be called the \emph{group algebra} associated with $G$. (Remark that the notation $\mathbb{C}[G]$ stands for $\mathcal{A}$ in \cite{Neshveyev} differently from here.)   Finally, we have the $C^*$-norm closures of $\bigodot_{\pi\in\widehat{G}}B(\mathcal{H}_\pi)$ and $\mathbb{C}[G]$ in $\bigoplus_{\pi\in\widehat{G}} B(\mathcal{H}_\pi)$ and $W^*(G)$, respectively, and the latter is denoted by $C^*(G)$, a non-unital $C^*$-algebra, which we call the \emph{group $C^*$-algebra} associated with $G$.  In what follows, we often identify $\mathbb{C}[G] \subset C^*(G) \subset W^*(G) \subset \mathcal{U}(G)$ with 
\[
\bigodot_{\pi\in\widehat{G}}B(\mathcal{H}_\pi) \subset \overline{\bigodot_{\pi\in\widehat{G}}B(\mathcal{H}_\pi)}^{\text{$C^*$-norm}} \subset \bigoplus_{\pi\in\widehat{G}} B(\mathcal{H}_\pi) \subset \prod_{\pi\in\widehat{G}} B(\mathcal{H}_\pi) 
\]
via the mapping $f\mapsto U_{\widehat{G}}(f)$. We remark that the above lemma guarantees that the action $\widehat{\tau}^G\colon\mathbb{R} \curvearrowright \mathcal{U}(G)$ naturally induces actions on $\mathbb{C}[G] \subset C^*(G) \subset W^*(G)$ of the real line $\mathbb{R}$ as its restrictions to them, and we still denote them by the same symbol. 

\medskip
We now introduce the notion of quantized characters of $G$.
\begin{definition}
A \emph{quantized character} of a CQG $G$ is a $\hat{\tau}^G$-KMS state on its group $C^*$-algebra $C^*(G)$ with inverse temperature $-1$ .
\end{definition}

Note that $C^*(G)$ is in general not unital, and hence a state on $C^*(G)$ is defined as a positive linear functional of norm $1$. In particular, the quantized characters do not form a compact set in general. The following lemma justifies this definition.

\begin{lemma}\label{Lemma:decom}
Any quantized character $\chi$ a CQG $G$ can be uniquely decomposed as \begin{equation}\label{Eq:decom}\chi=\sum_{\pi\in\widehat{G}}c_\pi\chi_\pi\quad (c_\pi\geq0),\end{equation} where $\chi_\pi$ is defined to be $\chi_\pi(\,\cdot\,):=\mathrm{Tr}(F_\pi p_\pi(\,\cdot\,))/\dim_q(\pi)$  and $p_\pi\colon C^*(G)\to B(\mathcal{H}_\pi)$ is the canonical projection. Conversely, a state on $C^*(G)$ defined by the right-hand side of Equation \eqref{Eq:decom} with $\sum_{\pi\in\widehat{G}}c_\pi=1$ becomes a quantized character of the CQG $G$.
\end{lemma}
\begin{proof}
Let $\iota_\pi\colon B(\mathcal{H}_\pi)\to C^*(G)$ be the canonical injective $*$-homomorphism and $\psi_\pi:=\chi\circ\iota_\pi$ on $B(\mathcal{H}_\pi)$. Since $B(\mathcal{H}_\pi)$ is finite-dimensional (and hence $x_\pi\in B(\mathcal{H}_\pi)\mapsto|\psi_\pi(x_\pi)|$ takes the maximum over the unit ball) and since $\|x\|=\sup_{\pi\in\widehat{G}}\|x_\pi\|$ for every $x=(x_\pi)_{\pi\in\widehat{G}}\in C^*(G)$, it is easy to show that \[\sum_{\pi\in J}\|\psi_\pi\|\leq\|\chi\|\] for every finite subset $J\subset\widehat{G}$. Thus we have $\sum_{\pi\in\widehat{G}}\|\psi_\pi\|\leq\|\chi\|$. Then we can define the bounded linear functional $\psi$ on $C^*(G)$ by \[\psi\colon x=(x_\pi)_{\pi\in\widehat{G}}\in C^*(G)\mapsto\sum_{\pi\in\widehat{G}}\psi_\pi(x_\pi),\] since \[\sum_{\pi\in\widehat{G}}|\psi_\pi(x_\pi)|\leq\sum_{\pi\in\widehat{G}}\|\psi_\pi\|\|x_\pi\|\leq\|x\|\sum_{\pi\in\widehat{G}}\|\psi_\pi\|\leq\|x\|\|\chi\|<\infty.\] It is easy to see that $\psi$ agrees with $\chi$ on $\mathbb{C}[G]$, and hence $\psi=\chi$ since $\mathbb{C}[G]$ is norm-dense in $C^*(G)$. Finally, we obtain that $\psi_\pi/\|\psi\|=\chi_\pi|_{B(\mathcal{H}_\pi)}$ because they are KMS states on $B(H_\pi)$ for the action $\widehat{\tau}^G|_{B(\mathcal{H}_\pi)}$ thanks to Lemma \ref{lemma:dualscalingaction_functional} and the uniqueness of KMS states on $B(\mathcal{H}_\pi)$ (see \cite[Exapmple 5.3.31]{BratteliRobinson2}).
\end{proof}

\subsection{Quantum subgroups and Their Duals}
Let $G=(A,\delta_G)$ be a CQG and $H=(B,\delta_H)$ a quantum subgroup of $G$ with surjective $*$-homomorphism $\theta\colon A\to B$. We fix complete families $\{U_\pi\}_{\pi\in\widehat{G}}$ and $\{U_\rho\}_{\rho\in\widehat{H}}$ of representatives.

\medskip
We will investigate the relation between the dual actions $\widehat{\tau}^G,\widehat{\tau}^H$ of the scaling groups $\tau^G,\tau^H$. To do so, we need to discuss the \emph{branching rule} of irreducible representations of $G,\, H$. Recall that for any $\pi \in\widehat{G}$ the restriction $U_{\pi|_H} = (\mathrm{id}\otimes\theta)(U_\pi)$ admits an irreducible decomposition, that is, there exist a finite subset $\mathcal{F}_\pi \subset \widehat{H}$, natural numbers $m_\pi(\rho)$, $\rho \in \mathcal{F}_\pi$, and $S_{\pi,l,\rho} \in \mathrm{Mor}(U_\rho,U_{\pi|_H})$, $1 \leq l \leq m_\pi(\rho)$, $\rho \in \mathcal{F}_\pi$, such that $S_{\rho,l,\pi}^*S_{\rho',l',\pi} = \delta_{\rho,\rho'}\delta_{l,l'}I_{\mathcal{H}_\rho}$ for every $1 \leq l \leq m_\pi(\rho)$ and $\rho \in \mathcal{F}_\pi$, $\sum_{\rho\in\mathcal{F}_\pi} \sum_{l=1}^{m_\pi(\rho)} S_{\rho,l,\pi}S_{\rho,l,\pi}^* = 1_{\mathcal{H}_\pi}$ and 
\begin{equation}\label{Eq4.1}
U_{\pi|_H} = \sum_{\rho\in\mathcal{F}_\pi}\sum_{l=1}^{m_\pi(\rho)} (S_{\rho,l,\pi}\otimes 1_B)U_\rho(S_{\rho,l,\pi}\otimes1_B)^*.
\end{equation} 
It is well known that the family $\mathcal{F}_\pi$ as well as the multiplicities $m_\pi(\rho)$ are uniquely determined and describe the brunching rule. We write $\rho \prec \pi$ for $(\rho,\pi) \in \widehat{H}\times\widehat{G}$, if $\rho \in \mathcal{F}_\pi$. The $S_{\rho,l,\pi}$ are not unique, and thus we choose and fix them throughout this section. 

\medskip
We define the map $\Theta:\mathcal{U}(H)\to\mathcal{U}(G)$ by $\Theta(f)=f\circ\theta$ for any $f\in\mathcal{U}(H)$. It is easy to see that the map $\Theta$ becomes a $*$-injective unital homomorphism. Note that the injectivity follows from the fact that $\theta$ sends the linear span of all matrix coefficients associated with $G$ onto that associated with $H$, see \cite[Lemma 2.8(1)]{Tomatsu}. We need the next simple proposition latter.

\begin{proposition}\label{P4.1} 
We have $\Theta(W^*(H)) \subset W^*(G)$. 
\end{proposition} 
\begin{proof} 
For any $\pi \in \widehat{G}$ and $f \in \mathcal{U}(H)$ we have 
\begin{equation}\label{Eq:sec4}
U_\pi(\Theta(f))=\sum_{\rho\in\widehat{H}; \rho \prec \pi} \sum_{l=1}^{m_\pi(\rho)} S_{\rho,l,\pi} U_\rho(f) S_{\rho,l,\pi}^*,
\end{equation}
thus, it follows that $\Vert U_\pi(\Theta(f))\Vert \leq \sup_{\rho \in \widehat{H};\rho\prec\pi} \Vert U_\rho(f)\Vert$. This immediately implies the desired assertion. 
\end{proof}

We remark that both $\Theta(\mathbb{C}[H]) \subset \Theta(C^*(H))$ do not sit inside $C^*(G)$ in general. This forces us to construct a canonical $C^*$-algebra associated with a given inductive system of CQGs in an indirect way; see the next section.

\medskip
Formula \eqref{Eq:sec4} gives an explicit description of the embedding $\Theta$. Indeed, we have \begin{equation}\label{eq:explicit_mapping}U_{\widehat{G}}(\Theta(f)) = \Big(\sum_{\rho \in \widehat{H}; \rho \prec \pi} \sum_{l=1}^{m_\pi(\rho)} S_{\rho,l,\pi}U_\rho(f)S_{\rho,l,\pi}^*\Big)_{\pi \in \widehat{G}},\quad f\in\mathcal{U}(H).\end{equation}

\medskip
We then investigate how $F_\pi=(\mathrm{id}\otimes f^G_1)(U_\pi),\pi\in\widehat{G}$ are related to $F_\rho=(\mathrm{id}\otimes f^H_1)(U_\rho),\rho\in\widehat{H}$.
\begin{lemma}\label{L4.1}
Let $\pi\in\widehat{G}$ be arbitrarily given. Then \[F_{\pi|_H}=(\mathrm{id}\otimes f^H_1)(U_{\pi|_H})=U_\pi(\Theta(f^H_1))=\sum_{\rho\in\widehat{H};\rho\prec\pi}\sum_{l=1}^{m_\pi(\rho)}S_{\rho,l,\pi}F_\rho S_{\rho,l,\pi}^*\in U_{\pi}(\Theta(\mathcal{U}(H))),\] and there exists a unique positive invertible element $W_\pi\in\mathrm{End}(U_{\pi|_H})=U_{\pi}(\Theta(\mathcal{U}(H)))'$ on $\mathcal{H}_\pi$ such that \begin{equation}\label{eq:density}F_\pi=W_\pi F_{\pi|_H}=F_{\pi|_H}W_\pi.\end{equation}
\end{lemma}
\begin{proof}
The first formula is trivial.

Since $U_{\pi|_H}^{cc}=(\mathrm{id}\otimes\theta)(U_\pi^{cc})$, we observe that $F_\pi\in\mathrm{Mor}(U_\pi,U_\pi^{cc})\subset\mathrm{Mor}(U_{\pi|_H},U_{\pi|_H}^{cc})$. Hence $W_\pi:=F_{\pi|_H}^{-1}F_\pi$ falls in $\mathrm{End}(U_{\pi|_H})$. We observe that $\mathrm{End}(U_{\pi|_H})=U_{\pi}(\Theta(\mathcal{H}))'$ on $\mathcal{H}_\pi$, and hence $F_\pi=F_{\pi|_H}W_\pi=W_\pi F_{\pi|_H}=F_{\pi|_H}^{1/2}W_\pi F_{\pi|_H}^{1/2}$ and $W_\pi=F_{\pi|_H}^{-1/2}F_\pi F_{\pi|_H}^{-1/2}$ is positive invertible. 
\end{proof}

\begin{remark}\label{R4.1}
Remark that the $S_{\rho,l,\pi}S_{\rho,m,\pi}^*,\ 1\leq l,m\leq m_\pi(\rho),\ \rho\prec\pi,$ form a matrix unit system of $\mathrm{End}(U_{\pi|_H})=U_\pi(\Theta(\mathcal{U}(H)))'$ on $\mathcal{H}_\pi$. The above lemma shows, in particular, that one can re-choose, by perturbing them by a suitable unitary in $U_\pi(\Theta(\mathcal{U}(H)))'$ for each $\pi$, a complete family $\{U_\pi\}_{\pi\in\widehat{G}}$ of representatives and the intertwiners $S_{\rho,l,\pi}$ \emph{with keeping $\{U_\rho\}_{\rho\in\widehat{H}}$ and $\{U_{\pi|_H}\}_{\pi\in\widehat{G}}$} in such a way that all $W_\pi$ are ``diagonalized", that is, \[W_\pi=\sum_{\rho\in\widehat{H};\rho\prec\pi}\sum_{l=1}^{m_\pi(\rho)}w(\rho,l,\pi)S_{\rho,l,\pi}S_{\rho,l,\pi}^*\] with $w(\rho,l,\pi)$ not decreasing (or not increasing) in $l$. In this case, the branching formula of the density matrix $F_\pi$ follows from Equation \eqref{eq:density}: \begin{equation}\label{Eq:Density}F_\pi=\sum_{\rho\in\widehat{H};\rho\prec\pi}\sum_{l=1}^{m_\pi(\rho)}w(\rho,l,\pi)S_{\rho,l,\pi}F_\rho S_{\rho,l,\pi}^*.\end{equation} This is an important observation to construct a canonical ``weighted branching graph" for a given inductive system of CQGs latter.
\end{remark}

The next proposition seems fundamental in the study of quantum subgroups and might be well-known to experts. Indeed, it follows from e.g. \cite[Proposition 3.15]{MeyerRoyWoronowicz}. However, we include the proof for the convenience of the reader.
\begin{proposition}\label{P4.2} 
We have $\theta\circ\tau_t^G = \tau_t^H\circ\theta$, and hence $\Theta\circ\widehat{\tau}^H_t = \widehat{\tau}^G_t\circ\Theta$ for every $t \in \mathbb{R}$. 
\end{proposition} 
\begin{proof}
It suffices to confirm the first identity against any matrix coefficients. For every $\pi \in \widehat{G}$ and $t \in \mathbb{R}$ we have, by formula \eqref{Eq:Modular_Scaling_Action} and Lemma \ref{L4.1},
\begin{align*}
(\mathrm{id}\otimes\theta\circ\tau^G_t)(U_\pi)
&= 
(F_\pi^{\sqrt{-1}t}\otimes1_B)U_{\pi|_H}(F_\pi^{-\sqrt{-1}t}\otimes1_B) \\
&=
\sum_{\rho\in\widehat{H}; \rho \prec \pi} \sum_{l=1}^{m_\pi(\rho)} 
(W_\pi^{\sqrt{-1}t}F_{\pi|_H}^{\sqrt{-1}t}S_{\rho,l,\pi}\otimes1_B)U_\rho(S_{\rho,l,\pi}^*F_{\pi|_H}^{-\sqrt{-1}t}W_\pi^{-\sqrt{-1}t}\otimes1_B) \\
&=
\sum_{\rho\in\widehat{H}; \rho \prec \pi} \sum_{l=1}^{m_\pi(\rho)} 
(W_\pi^{\sqrt{-1}t}S_{\rho,l,\pi}F_\rho^{\sqrt{-1}t}\otimes1_B)U_\rho(F_\rho^{-\sqrt{-1}t}S_{\rho,l,\pi}^*W_\pi^{-\sqrt{-1}t}\otimes1_B) \\
&=
(W_\pi^{\sqrt{-1}t}\otimes1_B)(\mathrm{id}\otimes\tau_t^H)(U_{\pi|_H})(W_\pi^{-\sqrt{-1}t}\otimes1_B) \\
&=
(\mathrm{id}\otimes\tau_t^H)(U_{\pi|_H}) =
(\mathrm{id}\otimes\tau_t^H\circ\theta)(U_\pi),
\end{align*}
where Lemma \ref{L4.1} guarantees that $F_\pi^{\sqrt{-1}t}=W_\pi^{\sqrt{-1}t} F_{\pi|_H}^{\sqrt{-1}t}$ as well as $(W_\pi^{\sqrt{-1}t}\otimes1_B)(\mathrm{id}\otimes\widehat{\tau}^H_t)(U_{\pi|_H})=(\mathrm{id}\otimes\widehat{\tau}^H_t)(U_{\pi|_H})(W_\pi^{\sqrt{-1}t}\otimes1_B)$. Hence we are done. 
\end{proof}

In closing of this section, we give an observation about the choice of the intertwiners $S_{\rho,l,\pi}$. Indeed, the next proposition shows that they are essentially unique (for our purpose), once the representatives $\{U_\rho\}_{\rho\in\widehat{H}}$ and $\{U_\pi\}_{\pi\in\widehat{G}}$ are fixed. The proof is not so hard and based on the same idea as in the proof of Lemma \ref{L4.1}. Hence we leave it to the reader.

\begin{proposition}\label{P:2.3}
The family of intertwiners $S_{\rho,l,\pi}$ are uniquely determined up to left multiplication of unitary elements in $\mathrm{End}(U_{\pi|_H})=U_\pi(\Theta(\mathcal{U}(H)))'$ on $\mathcal{H}_\pi$ by $U_{\pi|_H}$ and the $U_\rho$ with $\rho\in\widehat{H},\rho\prec\pi$.
\end{proposition}

\subsection{Stratila-Voiculescu AF-flows and Quantized Characters} \label{SVAF}

For a given inductive system of compact groups, Stratila and Voiculescu studied factor representations of its inductive limit group by using a certain AF-algebra, see e.g.\ \cite{StraVoic:book}. Following Enomoto--Izumi \cite{EnomotoIzumi} we call this AF-algebra the \emph{Stratila--Voiculescu AF-algebra}. In this section, we introduce the same kind of AF-algebra with a certain one-parameter automorphism group for a given inductive system of CQGs, which we call the \emph{Stratila--Voiculescu AF-flow}.

\medskip
Let $\mathbb{G} = (G_N,\theta_N)_{N=0}^\infty$ be an inductive system of CQGs, that is, each $G_N=(A_N,\delta_N)$ is a CQG and also a quantum subgroup of the next $G_{N+1}$ with the surjective $*$-homomorphism $\theta_N\colon A_{N+1}\to A_N$. We always assume that $\widehat{G}_N$ is countable for every $N\geq1$ and $G_0=(\mathbb{C},\mathrm{id}_\mathbb{C})$ throughout this paper. As we saw in the previous section, we have the inductive system $(W^*(G_N),\Theta_N)_{N=0}^\infty$ with injective unital $*$-homomorphisms $\Theta_N\colon W^*(G_N)\to W^*(G_{N+1})$. Hence we can take the $C^*$-inductive limit $\mathfrak{M}(\mathbb{G}) := \varinjlim_{N} (W^*(G_N),\Theta_N)$. Then we can faithfully embed all $C^*(G_N) \subset W^*(G_N)$ into $\mathfrak{M}(\mathbb{G})$ and denote by $\mathfrak{A}(\mathbb{G})$ the unital $C^*$-subalgebra generated by those $C^*(G_N)$ inside $\mathfrak{M}(\mathbb{G})$. Here we remark that $C^*(G_0) = W^*(G_0) = \mathbb{C}$. We call this unital $C^*$-algebra $\mathfrak{A}(\mathbb{G})$ the \emph{Stratila--Voiculescu AF-algebra} associated with the inductive system $\mathbb{G}$. Indeed, it immediately follows from \cite[Theorem 2.2]{Bratteli:AFpaper} that $\mathfrak{A}(\mathbb{G})$ is an AF-algebra. We remark that the inductive system $\Theta_N\colon W^*(G_N) \to W^*(G_{N+1})$ as well as the structure of $C^*(G_N) \subset W^*(G_N)$ are completely determined by the unitary duals $\widehat{G}_N$ and their branching rule. See Equation \eqref{eq:explicit_mapping} and Proposition \ref{P:2.3}. Therefore, the AF-algebra $\mathfrak{A}(\mathbb{G})$ itself never remembers the effect of ``$q$-deformation". However, the dual actions $\widehat{\tau}^{G_N}$ certainly remembers the $q$-deformation when one considers $q$-deformed classical groups like $U_q(N)$, and also they are known to be the modular actions associated with the dual Haar weights, see e.g.\ \cite{PodlesWoronowicz}. Fortunately, as we saw before, the dual actions $\widehat{\tau}^{G_N}$ are compatible with the inductive system $\Theta_N$ (see Proposition \ref{P4.2}), and hence we obtain the action $\widehat{\tau}^{\mathbb{G}}\colon\mathbb{R} \curvearrowright \mathfrak{A}(\mathbb{G})$ as the restriction of $\varinjlim_{N}\widehat{\tau}^{G_{N}}_t$, $t \in \mathbb{R}$ to $\mathfrak{A}(\mathbb{G})$. It is rather easy to see that $\widehat{\tau}^{\mathbb{G}}\colon\mathbb{R} \curvearrowright \mathfrak{A}(\mathbb{G})$ is pointwise norm continuous. In what follows, we call this action $\widehat{\tau}^{\mathbb{G}}\colon\mathbb{R} \curvearrowright \mathfrak{A}(\mathbb{G})$ the \emph{Stratila--Voiculescu AF-flow} associated with the inductive system $\mathbb{G}$. We regard a certain class of KMS states with respect to $\widehat{\tau}^{\mathbb{G}}$ as the characters of the inductive system $\mathbb{G}$ as follows.

\begin{definition} A \emph{quantized character} of the inductive system $\mathbb{G}$ is a $\widehat{\tau}^{\mathbb{G}}$-KMS state $\chi$ on $\mathfrak{A}(\mathbb{G})$ of inverse temperature $-1$ such that the restriction of $\chi|_{C^*(G_N)}$ to $C^*(G_N)$ is of norm $1$ for every $N = 1,2,\dots.$
\end{definition}

If the inductive system $\mathbb{G}$ comes from ordinary compact groups (or even compact Kac algebras), then the dual scaling actions $\widehat{\tau}^{G_N}$ are all trivial and hence so is the Stratila--Voiculescu AF-flow $\widehat{\tau}^{\mathbb{G}}$, implying that any quantized characters are tracial in the case. It is quite natural to assume that the restriction of $\chi|_{C^*(G_N)}$ to $C^*(G_N)$ is of norm $1$ for every $N=1,2,\dots$ in the above. In fact, this assumption clearly holds true in the case of characters of ordinary compact groups. Conversely, if a tracial state satisfies this assumption, then the norm of the restriction of the state to the closed ideal $J_N$ generated by $\bigcup_{k\geq N+1}C^*(G_k)$ in $\mathfrak{A}(\mathbb{G})$ is equal to $1$ for every $N=0,1,\dots$, and such a trace comes from some character on the inductive limit group, see the proof of \cite[Lemma 2.2(1)]{EnomotoIzumi}. 

\medskip
We denote the set of $\hat{\tau}^\mathbb{G}$-KMS states of inverse temperature $-1$ by $\mathrm{KMS}(\mathfrak{A}(\mathbb{G}),\hat{\tau}^\mathbb{G})$ and the set of quantized characters by $\mathrm{Ch}(\mathbb{G})$. The former is equipped with the topology of weak$^*$ convergence and the latter is equipped with the relative topology. Since the set $\mathrm{KMS}(\mathfrak{A}(\mathbb{G}),\hat{\tau}^\mathbb{G})$ is a Choquet simplex, see e.g.\ \cite[Thoerem 5.3.30(2)]{BratteliRobinson2}, the next proposition follows.

\begin{proposition}\label{Prop:IntegralRepresentation}
All of extremal points in $\mathrm{Ch}(\mathbb{G})$ are also extremal in $\mathrm{KMS}(\mathfrak{A}(\mathbb{G}),\hat{\tau}^\mathbb{G}),$ that is, \[\mathrm{ex}(\mathrm{Ch}(\mathbb{G}))=\mathrm{ex}(\mathrm{KMS}(\mathfrak{A}(\mathbb{G}),\hat{\tau}^\mathbb{G}))\cap\mathrm{Ch}(\mathbb{G}).\] Furthermore, for any quantized character $\chi\in\mathrm{Ch}(\mathbb{G})$ there exists a unique probability measure $M$ on the set of extremal points $\mathrm{ex}(\mathrm{Ch}(\mathbb{G}))$ such that \[\chi=\int_{\mathrm{ex}(\mathrm{Ch}(\mathbb{G}))}\epsilon\,dM(\epsilon),\] that is, \[\chi(a)=\int_{\mathrm{ex}(\mathrm{Ch}(\mathbb{G}))}\epsilon(a)\,dM(\epsilon)\] for any $a\in\mathfrak{A}(\mathbb{G})$.
\end{proposition}
\begin{proof}
The first statement is proved similarly to \cite[Lemma 2.2 (2)]{EnomotoIzumi} and the second statement is proved by Choquet--Mayer theorem, see \cite{Phelps:ChoquetTheorem}.
\end{proof}

For any $N=1,2,\dots$ and any $\pi\in\widehat{G}_N$, by Lemma \ref{lemma:dualscalingaction_functional}, we have $\widehat{\tau}^{G_N}_t|_{B(\mathcal{H}_\pi)}=\mathrm{Ad}(F_\pi^{\sqrt{-1}t})$. Thus, the state $\chi_\pi$ on $W^*(G_N)$ defined to be $\chi_\pi(x):=\mathrm{Tr}(F_\pi p_\pi x)/\mathrm{Tr}(F_\pi)$ is a KMS state on $W^*(G_N)$ for the action $\widehat{\tau}^{G_N}\colon\mathbb{R}\curvearrowright W^*(G_N)$, where $p_\pi$ is the projection onto $\mathcal{H_\pi}$. See \cite[Example 5.3.31]{BratteliRobinson2}. The following lemma is necessary in the next section and immediately follows from Lemma \ref{Lemma:decom}. 

\begin{lemma}\label{Decom}
For any $\chi\in\mathrm{KMS}(\mathfrak{A}(\mathbb{G}),\hat{\tau}^\mathbb{G})$ the restriction of $\chi$ to $C^*(G_N)$ is decomposed as \[\chi|_{C^*(G_N)}=\sum_{\pi\in\widehat{G}_N}c_\pi\chi_\pi,\quad c_\pi=\chi(p_\pi).\]
\end{lemma}

\subsection{Vershik-Kerov's ergodic method}\label{chap:DynamicalSystems}
In this section, we will establish the ergodic method for quantized characters of inductive system of CQGs. Firstly, we construct the (weighted) branching graph for a given inductive system. Then, we investigate dynamical systems on the paths on a given weighted branching graph. The goal of this section is Corollary \ref{Cor6.2}.

\subsubsection{Branching graphs and weighted central probability measures, weighted coherent systems}
We recall the concept of branching graphs (or Bratteli diagrams) and introduce a certain class of probability measures on their path spaces.

\begin{definition}
Let $\mathcal{G}:=(V,E,s,r)$ be a directed graph, that is, $V$ is the vertex set, $E$ is the edge set and $s,r:E\to V$ are the source and the range maps. The graph $\mathcal{G}$ is called a \emph{branching graph} if $V=\bigsqcup _{N=0}^\infty V_N,E=\bigsqcup_{N=1}^\infty E_N$, disjoint unions and $\mathcal{G}$ satisfies the following four conditions:
\begin{itemize} 
\item[(1)] $V_0$ consists of only one element denoted by $*$,
\item[(2)] $V_N$ is a countable set for every $N=1,2,\dots,$
\item[(3)] $|r^{-1}(\{v\})|<\infty$ for every $v\in V,$
\item[(4)] $s(E_N)=V_{N-1}$ and $r(E_N)=V_N$ for every $N=1,2,\dots$.
\end{itemize}
\end{definition}

Typical examples of branching graphs are Bratteli diagrams as well as the following example:
\begin{example}[The Gelfand--Tsetlin graph]\label{example:GelfandTsetlinGraph}
Our main example of branching graph is the Gelfand-Tsetlin graph $\mathbb{GT}$. This graph is associated with the branching rule of (quantum) unitary groups $U(N)$ ($U_q(N)$), see \cite{Zelobenko}, \cite{NoumiYamadaMimachi}. Let $\mathrm{Sign}_N$ be the set of \emph{signatures} of size $N$, that is, \[\mathrm{Sign}_N:=\{\nu=(\nu_1,\dots,\nu_N)\in\mathbb{Z}^N:\nu_1\geq\cdots\geq\nu_N\}.\] We set $\mathrm{Sign}_0:=\{*\}.$ For each $N\geq1,$ two signatures $\mu\in \mathrm{Sign}_N$ and $\nu\in \mathrm{Sign}_{N+1}$ are joined by an edge if and only if \[\nu_1\geq\mu_1\geq\nu_2\geq\cdots\geq\nu_N\geq\mu_N\geq\nu_{N+1}.\] We write $\mu\prec\nu$ in this case. Moreover, we assume that $*\in\mathrm{Sign}_0$ is joined to each vertex in $\mathrm{Sign}_1$ by only one edge. Define \[E^{GT}_N:=\{[\mu,\nu] : \mathrm{Sign}_{N-1}\ni\mu\prec\nu\in\mathrm{Sign}_N\}\] for each $N\geq1$ and \[\mathrm{Sign}:=\bigsqcup_{N\geq0}\mathrm{Sign}_N,\quad E^{GT}:=\bigsqcup_{N\geq1}E^{GT}_N.\] The source and the range maps $s,r:E^{GT}\to \mathrm{Sign}$ are defined as the projections onto the first and the second components, respectively. In this way, we have obtained the branching graph $\mathbb{GT}:=(\mathrm{Sign},E^{GT},s,r)$, called the \emph{Gelfand-Tsetlin graph}. Note that the number of the paths from $*$ to $\nu\in \mathrm{Sign}_N$ is exactly the dimension of the irreducible representation of $U(N)$ with label $\nu$. 
\end{example}

\begin{definition}[The branching graph associated with an inductive system of CQGs]
For each inductive system $\mathbb{G}=(G_N,\theta_N)_{N=0}^\infty$ with $G_0=(\mathbb{C},\mathrm{id})$ we can construct the branching graph arising from the branching rule of the irreducible representations in the following way: The vertex set $V_N,N=1,2,\dots,$ is exactly the unitary dual $\widehat{G}_N$. We set $V_0=\{*\}$, consisting only of the trivial representation. For any pair $(\rho,\pi)\in\widehat{G}_{N-1}\times\widehat{G}_N$, we define $m_\pi(\rho)$ to be the multiplicity of $\rho$ in $\pi|_{G_{N-1}}$. Then the edge set $E_N,N=1,2,\dots,$ is defined to be $\{[\rho,l,\pi]:1\leq l\leq m_\pi(\rho),\widehat{G}_{N-1}\ni\rho\prec\pi\in\widehat{G}_N\}$. The source and the range maps $s,r\colon E:=\bigsqcup_{N\geq1}E_N\to V:=\bigsqcup_{N\geq0}V_N$ are the projections to the first and the third coordinates, respectively. Then, the quadlet $(V,E,s,r)$ is a branching graph. When all the multiplicities $m_\pi(\rho)$ are equal to $1$, we simply denote each edge by $[\rho,\pi]$. 
\end{definition}

We remark that the Gelfand--Tsetlin graph is the branching graph associated with the inductive system of $U_q(N)$.

\medskip
For any $K<N$ a sequence $(t_n)_{n=K}^N\in\prod_{n=K}^NE_n$ of edges is called a \emph{path} on a branching graph $\mathcal{G}$ if $r(t_n)=s(t_{n+1})$ for every $n=K,\dots,N-1.$ When $N<\infty$, the path is called a finite path; otherwise it is called an infinite path. For any $u\in V_K,v\in V_N$ with $K<N$ we denote by $\Omega(u,v)$ the set of all (finite) paths from $u$ to $v$, that is, the source of the first edge is $u$ and the range of the final edge is $v$. We define $\dim(u,v):=|\Omega(u,v)|$, the number of elements of $\Omega(u,v)$, called the \emph{relative dimension} of $u$ with respect to $v$. When $u=*\in V_0$, we call the \emph{dimension} of $v$ and denote it by $\dim(v)$. When the branching graph $\mathcal{G}$ associated with an inductive system of CQGs, the dimension of each vertex is nothing but the dimension of the corresponding irreducible representation. 

\medskip
We introduce the notion of weighted branching graphs. Let $\mathcal{G}=(V,E,s,r)$ be a branching graph. A \emph{weight function} is a function $w\colon E\to(0,\infty)$, and the branching graph $\mathcal{G}$ equipped with a weight function $w$ is called a \emph{weighted branching graph}. For any finite path $t=(t_n)_{n=K}^N,$ its weight $w(t)$ is defined to be $w(t_K)w(t_{K+1})\cdots w(t_N)$ and the \emph{weighted dimension} of each vertex $v$ is defined to be $\sum_{t\in\Omega(*,v)}w(t)$, denoted by $w$-$\dim(v)$. Moreover, for any $u\in V_K,v\in V_N$ with $K<N$ the relative weighted dimension from $u$ to $v$ is similarly defined to be $\sum_{\omega\in\Omega(u,v)}w(\omega)$ and denoted by $w$-$\dim(u,v).$ 

\begin{definition}[The weighted branching graph associated with an inductive system of CQGs]
On the branching graph $\mathcal{G}$ associated with an inductive system $\mathbb{G}$ of CQGs, we can construct a canonical weight function as follows. By Remark \ref{R4.1} we can inductively select the sequence of families of representatives, say $\{U_\pi\}_{\pi\in\widehat{G}_N},N=1,2,\dots$ with intertwiners $S_{\rho,l,\pi}$ with $(\rho,\pi)\in\widehat{G}_{N-1}\times\widehat{G}_N,\rho\prec\pi,1\leq l\leq m_\pi(\rho),$ in such a way that \[W_\pi=\sum_{\rho\in\widehat{G}_{N-1};\rho\prec\pi}\sum_{l=1}^{m_\pi(\rho)}w(\rho,l,\pi)S_{\rho,l,\pi}S_{\rho,l,\pi}^*,\quad \pi\in\widehat{G}_N\] with $w(\rho,l,\pi)$ not decreasing in $l$ for every $N=1,2,\dots$. These $w(\rho,l,\pi)$ define a weight function $w$, i.e., $w([\rho,l,\pi]):=w(\rho,l,\pi)$. In this way, we can define a canonical weight function $w$ on $\mathcal{G}$. Since the list $w(\rho,l,\pi)$ is nothing but the eigenvalues of $W_\pi$, this weight function is essentially independent of the choice of intertwiners $S_{\rho,l,\pi}$.  
\end{definition}

\begin{remark}\label{R6.1}
By Formula \eqref{Eq:Density}, we have \[w([\rho,l,\pi])=\frac{\mathrm{Tr}(S^*_{\pi,l,\rho}F_\pi S_{\pi,l,\rho})}{\mathrm{Tr}(F_\rho)}\] for every edge $[\rho,l,\pi]\in E_N,N=1,2,\dots$. Let $S_e:=S_{\rho,l,\pi}$ for every edge $e=[\rho,l,\pi]$ and $S_t:=S_{t_N}S_{t_{N-1}}\cdots S_{t_1}$ for every finite path $t=(t_n)_{n=1}^N$. By Formula \eqref{Eq:Density} again we have \begin{equation}\label{Eq:R6.1}F_\pi=\sum_{t\in\Omega(*,\pi)}w(t)S_tS_t^*.\end{equation} The $S_tS_u^*, t,u\in\Omega(*,\pi)$ form a system of matrix units on $B(\mathcal{H}_\pi)$. 
\end{remark}

\medskip
Let $\Omega:=\Omega(\mathcal{G})$ be the space of infinite paths starting at $*$ on a branching graph $\mathcal{G}$. For any finite path $t:=(t_n)_{n=1}^N$ (starting at $*$), \emph{the cylinder set} $C_t$ is defined by \[C_t:=\{(\omega_n)_{n=1}^\infty\in\Omega:\omega_n=t_n,n=1,\dots,N\}.\] Let $\mathcal{F}$ be the $\sigma$-algebra generated by the collection of cylinder sets. It is easy to see that the measurable space $(\Omega,\mathcal{F})$ is a unique standard Borel space. Following Gorin's definition, see \cite[Section 1.2]{Gorin:qcentralmeasure}, a probability measure $P$ on $(\Omega,\mathcal{F})$ is called \emph{$w$-central} if the following holds: For any finite path $t$ starting at $*$ and terminating at $v\in V_N,N=1,2,\dots$ \[\frac{P(C_t)}{w(t)}=\frac{P(X_N=v)}{w\mathchar`-\dim(v)},\] where the measurable function $X_N:(\Omega,\mathcal{F})\to V_N$ is defined by $X_N(\omega):=r(\omega_N).$ The set of $w$-central probability measures is denoted by $\mathrm{Cent}(\mathcal{G},w)$, and it is clearly a convex set. A sequence of probability measures $P_N$ on $V_N$ with $N=1,2,\dots$ is called a $w$-\emph{coherent system} if the following coherent relation holds: \begin{equation}\label{CoherentRelation}P_N(v)=\sum_{v'\in V_{N+1}}\Bigg(\sum_{e\in\Omega(v,v')}w(e)\Bigg)\frac{w\mathchar`-\dim(v)}{w\mathchar`-\dim(v')}P_{N+1}(v')\end{equation} for any $v\in V_N$ and $N=1,2,\dots.$ The set of $w$-coherent systems is denoted by $\mathrm{Coh}(\mathcal{G},w)$, and it is also a convex set. When a weighted branching graph $(\mathcal{G},w)$ is associated with an inductive system $\mathbb{G}$, we denote the set $\mathrm{Coh}(\mathcal{G},w)$ by $\mathrm{Coh}(\mathbb{G})$, also the set $\mathrm{Cent}(\mathcal{G},w)$ is denoted by $\mathrm{Cent}(\mathbb{G})$. There exists an affine bijection between $\mathrm{Coh}(\mathcal{G},w)$ and $\mathrm{Cent}(\mathcal{G},w)$. The bijection is simply given by $P_N(v)=P(X_N=v)$ for any $v\in V_N,N=1,2,\dots$. We consider the topology of weak convergence on $\mathrm{Cent}(\mathcal{G},w)$ and the topology of component-wise weak convergence on $\mathrm{Coh}(\mathcal{G},w)$. The next proposition trivially holds true.

\begin{proposition}\label{P6.1}
The convex sets $\mathrm{Cent}(\mathcal{G},w)$ and $\mathrm{Coh}(\mathcal{G},w)$ are (affine-)homeomorphic by the correspondence $P\mapsto(P_N)_{N=1}^\infty$ with $P_N(v)=P(X_N=v)$ for any $v\in V_N,N=1,2,\dots$.
\end{proposition}

\medskip
Finally, we introduce a certain group of measurable transformations on $(\Omega,\mathcal{F})$. For any $v\in V_N,N=1,2,\dots$, the group of all permutations of $\Omega(*,v)$ is denoted by $\mathfrak{S}_v^0$. It is naturally embedded into the measurable transformations on $(\Omega,\mathcal{F})$, that is, for any $\gamma\in\mathfrak{S}_v^0$ and any $\omega=(\omega_n)_n\in\Omega$, we define by \[\gamma(\omega):=\begin{cases}(\gamma(\omega_1,\dots,\omega_N),\omega_{N+1},\dots)&(r(\omega_N)=v)\\\omega&\text{(otherwise).}\end{cases}\]  Let $\mathfrak{S}_v(\Omega)$ be the embedding of $\mathfrak{S}^0_v$ and define $\mathfrak{S}_N(\Omega)$ to be the subgroup of all the measurable transformations on $\Omega$ generated by $\bigcup_{v\in V_N}\mathfrak{S}_v(\Omega)$ which is isomorphic to $\bigoplus_{v\in V_N}\mathfrak{S}_v(\Omega)$ as abstract groups. Trivially, the group $\mathfrak{S}_N(\Omega)$ is a subgroup of $\mathfrak{S}_{N+1}(\Omega)$, and hence we obtain the transformation group $\mathfrak{S}(\Omega):=\varinjlim_{N}\mathfrak{S}_N(\Omega)(=\bigcup_{N\geq1}\mathfrak{S}_N(\Omega))$ on $\Omega$. For any $\gamma\in\mathfrak{S}(\Omega)$, any cylinder set $C$ and any $w$-central probability measure $P$ it follows that $P(C)=0$ if and only if $P(\gamma(C))=0$ from the definition of $w$-central probability measures. Thus $w$-central probability measures are quasi-invariant under the transformation group $\mathfrak{S}(\Omega)$. Later, we will show that when a weighted branching graph arises from an inductive system $\mathbb{G}$ of CQGs, a $w$-central probability measure $P$ is $\mathfrak{S}(\Omega)$-ergodic if and only if $P$ is extremal in $\mathrm{Cent}(\mathbb{G})$ (see Theorem \ref{Theorem:Ergodic_ExtremalKMS}).

\subsubsection{$w$-central probability measures and quantized characters}
For an inductive system $\mathbb{G}=(G_N,\theta_N)_{N=0}^\infty$ of CQGs with $G_0=(\mathbb{C},\mathrm{id})$, we will investigate the relation between the set $\mathrm{Ch}(\mathbb{G})$ of quantized characters and the set $\mathrm{Cent}(\mathbb{G})$ of $w$-central probability measures (or the set $\mathrm{Coh}(\mathbb{G})$ of $w$-coherent systems).

\begin{proposition}\label{P6.2}
There exists an affine homeomorphism between $\mathrm{Ch}(\mathbb{G})$ and $\mathrm{Coh}(\mathbb{G})$ such that the quantized character $\chi$ corresponding to a $w$-coherent system $(P_N)_N$ is decomposed as \[\chi|_{C^*(G_N)}=\sum_{\pi\in\widehat{G}_N}P_N(\pi)\chi_\pi\] on the $C^*$-subalgebra $C^*(G_N)$. Furthermore, there exists an affine homeomorphism between $\mathrm{Ch}(\mathbb{G})$ and $\mathrm{Cent}(\mathbb{G})$ such that the quantized character $\chi$ corresponding to a $w$-central probability measure $P$ is decomposed as \[\sum_{\pi\in\widehat{G}_N}P(X_N=\pi)\chi_\pi\] on the $C^*$-subalgebra $C^*(G_N),$ where $X_N$ is the measurable function defined by $X_N(\omega)=r(\omega_N)$ for any $\omega=(\omega_n)_{n=1}^\infty\in\Omega$.
\end{proposition}
\begin{proof}
For a given $w$-coherent system $(P_N)_N$, the KMS state $\chi_N$ on $W^*(G_N)$ is defined to be \[\chi_N:=\sum_{\pi\in\widehat{G}_N}P_N(\pi)\chi_\pi.\] We define $p_\pi$ to be the projection onto $\mathcal{H}_\pi$. Since the net $a(\mathcal{S}):=(a_\pi(\mathcal{S}))_{\pi\in\widehat{G}}$ with a finite subset $\mathcal{S}$ of $\widehat{G}_N$, where $a_\pi(\mathcal{S})$ is defined by \[a_\pi(\mathcal{S}):=\begin{cases}p_\pi&(\pi\in\mathcal{S})\\0&(\pi\not\in\mathcal{S})\end{cases},\] is an approximate unit of $C^*(G_N)$, we have \[\|\chi_N\|=\lim_{\mathcal{S}\to\widehat{G}_N}\chi_N(a(\mathcal{S}))=\lim_{\mathcal{S}\to\widehat{G}_N}\sum_{\pi\in\mathcal{S}}P_N(\pi)=1.\] See e.g.\, \cite[Lemma I.9.5]{Davidson}. For every $x=(x_\rho)_{\rho\in\widehat{G}_N}\in W^*(G_N)$, by Formulas \eqref{Eq:sec4} and \eqref{Eq:Density}, we have \begin{align*}\chi_{N+1}(\Theta_N(x))&=\sum_{\pi\in\widehat{G}_{N+1}}P_{N+1}(\pi)\chi_\pi\Big(\sum_{\rho\in\widehat{G}_N}\sum_{e\in\Omega(\rho,\pi)}S_ex_\rho S_e^*\Big)\\&=\sum_{\pi\in\widehat{G}_{N+1}}P_{N+1}(\pi)\sum_{\rho\in\widehat{G}_N;\rho\prec\pi}\Big(\sum_{e\in\Omega(\rho,\pi)}w(e)\Big)\frac{w\mathchar`-\dim(\rho)}{w\mathchar`-\dim(\pi)}\chi_\rho(x_\rho)\\&=\sum_{\rho\in\widehat{G}_N}\Big(\sum_{\pi\in\widehat{G}_{N+1};\rho\prec\pi}\frac{w\mathchar`-\dim(\rho)}{w\mathchar`-\dim(\pi)}\sum_{e\in\Omega(\rho,\pi)}P_{N+1}(\pi)\Big)\chi_\rho(x_\rho)\\&=\chi_N(x),\end{align*} where we can freely change the order of sums since the above summations absolutely converge. Hence we can construct the state $\chi$ on $\mathfrak{A}(\mathbb{G})$ by $\chi|_{C^*(G_N)}=\chi_N$ on $C^*(G_N)$. Recall $\bigcup_{N\geq0} \mathbb{C}[G_N]$ is a norm-dense in $\mathfrak{A}(\mathbb{G})$ and $\widehat{\tau}^\mathbb{G}$-invariant. By Lemma \ref{lemma:dualscalingaction_functional}, Formulas \eqref{Eq:sec4} and \eqref{Eq:Density}, we have $\chi(x\widehat{\tau}^\mathbb{G}_{-\sqrt{-1}}(y))=\chi(yx)$ for any $x,y\in\bigcup_{N\geq0} \mathbb{C}[G_N]$. Thus, by \cite[Section 5.3.1]{BratteliRobinson2}, $\chi$ is a $\widehat{\tau}^\mathbb{G}$-KMS state. Since the norm of the restriction $\chi$ to $C^*(G_N)$ is equal to $1$ for every $N\geq1$, the KMS state $\chi$ falls in $\mathrm{Ch}(\mathbb{G})$. Thus, we have the desired map from $\mathrm{Coh}(\mathbb{G})$ to $\mathrm{Ch}(\mathbb{G})$. 

Next, we consider the inverse map. By Lemma \ref{Decom}, every quantized character $\chi\in\mathrm{Ch}(\mathbb{G})$ can be decomposed as \[\chi|_{C^*(G_N)}=\sum_{\pi\in\widehat{G}_N}c_\pi\chi_\pi\] on $C^*(G_N)$ for any $N=1,2,\dots$ with non-negative coefficients $c_\pi$. Since the restriction $\chi|_{C^*(G_N)}$ to $C^*(G_N)$ is of norm $1$, we have \[\sum_{\pi\in\widehat{G}_N}c_\pi=\lim_{\mathcal{S}\to\widehat{G}_N}\chi(a(\mathcal{S}))=1,\] that is, the function $P_N$ defined to be $P_N(\pi):=c_\pi$ is a probability measure on $\widehat{G}_N$. It suffices to show that the sequence of probability measures $(P_N)_N$ becomes a $w$-coherent system. Indeed, for any $\rho\in\widehat{G}_{N}$, by Formula \eqref{Eq:Density}, we have  \begin{align*}P_{N}(\rho)&=\chi_{N}(p_\rho)\\&=\chi_{N+1}(\Theta_{N}(p_\rho))\\&=\sum_{\pi\in\widehat{G}_{N+1};\rho\prec\pi}P_{N+1}(\pi)\chi_{\pi}\Big(\sum_{l=1}^{m_\pi(\rho)}S_{\rho,l,\pi}S_{\rho,l,\pi}^*\Big)\\&=\sum_{\pi\in\widehat{G}_{N+1};\rho\prec\pi}\Bigg(\sum_{e\in\Omega(\rho,\pi)}w(e)\Bigg)\frac{w\mathchar`-\dim(\rho)}{w\mathchar`-\dim(\pi)}P_{N+1}(\pi),\end{align*} and thus the sequence $(P_N)_N$ satisfies the coherent relation \eqref{CoherentRelation}. 
\end{proof}

From this correspondence and Proposition \ref{Prop:IntegralRepresentation}, the next unique integral representation theorem for $w$-central probability measures follows.
\begin{theorem}\label{T6}
For any $w$-central probability measure $P\in\mathrm{Cent}(\mathbb{G})$, there exists a unique Borel probability measure $m$ on $\mathrm{Cent}(\mathbb{G})$ supported on the extremal points $\mathrm{ex}(\mathrm{Cent}(\mathbb{G}))$ such that \[P=\int_{\mathrm{ex}(\mathrm{Cent}(\mathbb{G}))}Q\,dm(Q).\] 
\end{theorem}
Here we remark that the theorem was already proved by Gorin, see \cite[Proposition 5.18]{Gorin:qcentralmeasure}. However, we gave a new approach to the unique integral representation theorem. We would like to emphasize that the approach here is quite natural in view of the original work due to Vershik--Kerov \cite{Kerov:book}.

\subsubsection{Krieger constructions and GNS constructions}\label{Krieger_constructions}
We assume that a weighted branching graph $(\mathcal{G},w)$ is associated with an inductive system $\mathbb{G}=(G_N,\theta_N)_N$ of CQGs with $G_0=(\mathbb{C},\mathrm{id}_\mathbb{C})$. In the section, we will show the following theorem:

\begin{theorem}\label{Theorem:Ergodic_ExtremalKMS}
A $w$-central probability measure $P$ is $\mathfrak{S}(\Omega)$-ergodic if and only if the corresponding quantized character $\chi\in\mathrm{Ch}(\mathbb{G})$ is factorial, namely, it is extremal. Therefore, $P$ is $\mathfrak{S}(\Omega)$-ergodic if and only if $P$ is extremal.
\end{theorem}

This equivalence has been known when the branching graph comes from an ordinary group. Indeed, $w$-central probability measures coincide with $\mathfrak{S}(\Omega)$-invariant probability measures in this case, and the equivalence follows from Choquet's theory (see \cite{Phelps:ChoquetTheorem}). A representation-theoretic interpretation of this fact is as follows. Characters of an inductive limit group (or tracial states on the associated Stratila--Voiculescu AF-algebra) correspond to central probability measures on the paths on the branching graph, and extremality for characters (or tracial states) (or the ergodicity for probability measures) corresponds to the factoriality for the corresponding representations of the inductive limit group. Here we prove Theorem \ref{Theorem:Ergodic_ExtremalKMS} by identifying the GNS representation of $\mathfrak{A}(\mathbb{G})$ associated with a quantized character $\chi$ with a kind of crossed product algebra associated with a certain dynamical system on $\Omega$ equipped with the $w$-central probability measure $P$ corresponding to $\chi$. This is completely analogous to the study of AF (or LS)-algebras due to Vershik and Kerov, see \cite{VershikKerov:AF},\cite{Kerov:book}.

\medskip
For the given $w$-central probability measure $P$, we will construct the von Neumann algebra of the dynamical system $(\Omega,\mathcal{F},P,\mathfrak{S}(\Omega))$. The equivalence relation $\mathcal{R}$ is defined by \[\mathcal{R}:=\{(\omega,\omega')\in\Omega\times\Omega:\exists\gamma\in\mathfrak{S}(\Omega), \omega'=\gamma(\omega)\}\] and it is called the \emph{tail equivalence relation}. We denote its equivalent classes by $[\cdot ]$. The projection $\mathrm{pr}:\mathcal{R}\to\Omega$ is defined by $\mathrm{pr}(\omega,\omega'):=\omega$ for any $(\omega,\omega')\in\mathcal{R}$. It is known that the set function $P_l$ on the $\sigma$-algebra $\mathcal{B}:=(\mathcal{F}\times\mathcal{F})\cap\mathcal{R}$ defined by \[P_l(A):=\int_\Omega|A\cap\mathrm{pr}^{-1}(\omega)|\,dP(\omega),\quad A\in\mathcal{B}\] becomes a measure on $(\mathcal{R},\mathcal{F})$, see \cite[Theorem 2]{FeldmanMoore1}. The measure $P_l$ is called the \emph{left measure} of $P$ on $\mathcal{R}$. The following space of bounded ``small support'' functions \[K(\mathcal{R}):=\{f\in L^\infty(\mathcal{R},P_l):  \mathrm{ess.sup}|\{\omega':f(\omega,\omega')\neq0\}|<\infty,\ \mathrm{ess.sup}|\{\omega:f(\omega,\omega')\neq0\}|<\infty\}\] becomes a $*$-algebra, whose multiplication and $*$-operation are defined by \[fg(\omega,\omega'):=\sum_{\omega''\in[\omega]}f(\omega,\omega'')g(\omega'',\omega'),\]\[f^*(\omega,\omega'):=\overline{f(\omega',\omega)}\] for any $f,g\in K(\mathcal{R})$ and $(\omega,\omega')\in\mathcal{R}$. The $*$-representation $\varpi_l$ of $K(\mathcal{R})$ on the Hilbert space $L^2(\mathcal{R},\mathcal{B},P_l)$ is defined by \[[\varpi_l(f)\xi](\omega,\omega'):=\sum_{\omega''\in[\omega]}f(\omega,\omega'')\xi(\omega'',\omega')\] for any $f\in K(\mathcal{R})$ and $\xi\in L^2(\mathcal{R},\mathcal{B},P_l)$. The von Neumann algebra $W^*(\mathcal{R};P)$ is defined to be the double commutant $\varpi_l(K(\mathcal{R}))''$, that is, the closure with respect to the strong operator topology (see e.g.\ \cite[Theorem I.7.1]{Davidson}). This construction is called the \emph{Krieger construction}. It is well known that the von Neumann algebra $W^*(\mathcal{R};P)$ is a factor if and only if $P$ is $\mathfrak{S}(\Omega)$-ergodic, see \cite[Proposition 2.9(2)]{FeldmanMoore2}.

\medskip
Next, we construct another von Neumann algebra which also acts on $L^2(\mathcal{R},\mathcal{B},P_l)$. The equivalence relation $\mathcal{R}_N$ is defined by \[\mathcal{R}_N:=\{(\omega,\omega')\in\Omega\times\Omega:\exists\gamma\in\mathfrak{S}_N(\Omega),\omega'=\gamma(\omega)\}\] and its equivalence classes are denoted by $[\,\cdot\,]_N$. For any $\rho\in\widehat{G}_N$ and $t,u\in\Omega(*,\rho)$ we define $f_{t,u}$ to be the characteristic function of $(C_t\times C_u)\cap\mathcal{R}_N$. Remark that $f_{t,u}\in K(\mathcal{R})$ by the definition. Every pair $(\omega,\omega')\in(C_t\times C_u)\cap\mathcal{R}_N$ can be written as $(\omega,\omega')=((t,\omega_{N+1},\omega_{N+2},\dots),(u,\omega_{N+1},\omega_{N+2},\dots))$. Using this, for any $t,u\in\Omega(*,\rho),t',u'\in\Omega(*,\rho')$ and any $(\omega,\omega')\in\mathcal{R}$ we have \begin{align*}f_{t,u}f_{t',u'}(\omega,\omega')&=\sum_{\omega''\in[\omega]}f_{t,u}(\omega,\omega'')f_{t',u'}(\omega'',\omega')\\&=\delta_{\rho,\rho'}\delta_{u,t'}1_{\mathcal{R}_N}(\omega,\omega')1_{C_t}(\omega)1_{C_{u'}}(\omega')=\delta_{\rho,\rho'}\delta_{u,t'}f_{t,u'}(\omega,\omega')\end{align*} and \[f_{t,u}^*(\omega,\omega')=f_{t,u}(\omega',\omega)=f_{u,t}(\omega,\omega').\] Thus the functions $f_{t,u},\ t,u\in\Omega(*,\rho),\rho\in\widehat{G}_N$ form a matrix unit system and we obtain the $*$-homomorphism $\varrho_N:W^*(G_N)\to W^*(\mathcal{R};P)$ defined by $\varrho_N(S_tS_u^*):=\varpi_l(f_{t,u})\in\varpi_l(K(\mathcal{R}))$, where $S_tS_u^*$ is a matrix unit of $B(\ell^2(\Omega(*,\rho)))$, see Remark \ref{R6.1}. 

\begin{lemma}
We have $\varrho_{N+1}\circ\Theta_N=\varrho_N$ on $W^*(G_N)$ for all $N=1,2,\dots$.
\end{lemma}
\begin{proof}
It suffices to show that $\varrho_{N+1}(\Theta_N(S_tS_u^*))=\varrho_N(S_tS_u^*)=\varpi_l(f_{t,u})$ for any $t,u\in\Omega(*,\rho)$ and any $\rho\in\widehat{G}_N$, where $S_tS_u^*$ is a matrix unit, see Remark \ref{R6.1}. By Formula \eqref{Eq:sec4}, we have \begin{align*}\varrho_{N+1}(\Theta_N(S_tS_u^*))&=\sum_{\pi\in\widehat{G}_{N+1};\rho\prec\pi}\sum_{e\in\Omega(\rho,\pi)}\rho_{N+1}(S_eS_tS_u^*S_e^*)\\&=\sum_{\pi\in\widehat{G}_N;\rho\prec}\sum_{e\in\Omega(\rho,\pi)}\varpi_l(f_{(t,e),(u,e)})\\&=\varpi_l(f_{t,u})=\varrho_N(S_tS_u^*).\end{align*} Hence we are done.
\end{proof}

By the lemma, we obtain the representation $\varrho$ of $\mathfrak{A}(\mathbb{G})$ on $L^2(\mathcal{R},\mathcal{B},P_l)$ defined by $\varrho|_{C^*(G_N)}=\varrho_N$ for any $N=1,2,\dots$. By the construction, the double commutant $\varrho(\mathfrak{A}(\mathbb{G}))''$ is a von Neumann subalgebra of $W^*(\mathcal{R};P)$. Furthermore, we have the following:

\begin{theorem}\label{T6.3}
The above two von Neumann algebras coincide, that is, $\varrho(\mathfrak{A}(\mathbb{G}))''=W^*(\mathcal{R};P)$. 
\end{theorem}
\begin{proof}
It suffices to show that $\varpi_l(K(\mathcal{R}))\subset\varrho(\mathfrak{A}(\mathbb{G}))''$. For any function $f\in L^\infty(\Omega,\mathcal{F},P)$ the function $\tilde{f}$ on $\mathcal{R}$ is defined to be $\tilde{f}(\omega,\omega'):=\delta_{\omega,\omega'}f(\omega)$. For any measurable bijection $\phi\colon(\Omega,\mathcal{F})\to(\Omega,\mathcal{F})$ such that $\{(\omega,\phi(\omega)):\omega\in\Omega\}\subset\mathcal{R}$, the function $F_\phi$ on $\mathcal{R}$ is defined to be \[F_\phi(\omega,\omega'):=\begin{cases}1&(\omega=\phi(\omega'))\\0&\text{(otherwise).}\end{cases}\] By \cite[Proposition 2.4]{FeldmanMoore2}, every function in $K(\mathcal{R})$ is a finite linear combination of functions of the form $\tilde{f}F_\phi$ with a function $f$ and a measurable bijection $\phi$ as above. Thus, our goal is to prove that $\varpi_l(f),\varpi(F_\phi)\in\varrho(\mathfrak{A}(\mathbb{G}))''$ for any $f\in L^\infty(\Omega,\mathcal{F},P)$ and any measurable bijection $\phi\colon(\Omega,\mathcal{F})\to(\Omega,\mathcal{F})$ such that $\{(\omega,\phi(\omega)):\omega\in\Omega\}\subset\mathcal{R}$.

For any finite path $t$ we have $\varpi_l(1_{C_t})=\varpi_l(f_{t,t})=\varrho(S_tS_t^*)\in\varrho(\mathfrak{A}(\mathbb{G}))\subset\varrho(\mathfrak{A}(\mathbb{G}))''$. Note that the collection of all cylinder sets and the empty set is a $\pi$-system and the collection of the set $A$ such that $1_A\in\varrho(\mathfrak{A}(\mathbb{G}))''$ is a $\lambda$-system. Thus, by the $\pi$-$\lambda$ theorem, we have $\varpi_l(1_A)\in\varrho(\mathfrak{A}(\mathbb{G}))''$ for any $A\in\mathcal{F}$. For any $f\in L^\infty(\Omega,\mathcal{F},P)$, there exists a sequence $(\psi_n)_{n\geq1}$ of simple functions such that $\psi_n\to f$ as $n\to\infty$ $P$-a.s. with $0\leq|\psi_1|\leq|\psi_2|\leq\cdots\leq |f|$. Then $\varpi_l(\widetilde{\psi}_n)\in\varrho(\mathfrak{A}(\mathbb{G}))$ and $\|\varpi_l(\widetilde{\psi}_n)\|\leq\|\varpi_l(\tilde{f})\|$ for any $n\geq1$. Moreover, for any $\eta\in L^2(\mathcal{R},\mathcal{B},P_l)$, by the dominated convergence theorem, \[\|\varpi_l(\tilde{f}-\widetilde{\psi}_n)\eta\|^2=\int_\mathcal{R}|(f(\omega)-\psi_n(\omega))\eta(\omega,\omega')|^2\,dP_l(\omega,\omega')\to0\quad\text{as }n\to\infty,\] that is, $\varpi_l(\widetilde{\psi}_n)$ converges to $\varpi_l(\tilde{f})$ in the strong operator topology. Thus, $\varpi_l(\tilde{f})\in\varrho(\mathfrak{A}(\mathbb{G}))''$. 

For any measurable bijection $\phi\colon(\Omega,\mathcal{F})\to(\Omega,\mathcal{F})$ such that $\{(\omega,\phi(\omega)):\omega\in\Omega\}\subset\mathcal{R}$, by \cite[Remark of Proposition 3.3]{FeldmanMoore1}, there exists $g_n\in\mathfrak{S}(\Omega)$ and $ D_n\in\mathcal{B}$ with $n=1,2,\dots$ such that the collection $\{D_n\}_{n\geq1}$ is a partition of $\mathcal{R}$ and $D_n\subseteq\{(\omega,\phi(\omega)):\phi(\omega)=g_n(\omega)\}$. Then we have $F_\phi=\sum_{n\geq1}\widetilde{1}_{g_n(D_n)}F_{g_n}$. Remark that $\varpi_l(F_g)\in\varrho(\mathfrak{A}(\mathbb{G}))\subset\varrho(\mathfrak{A}(\mathbb{G}))''$ for any $g\in\mathfrak{S}(\Omega)$. Thus, if the partition $\{D_n\}_n$ is finite, then the operator $\varpi_l(F_\phi)$ belongs to $\varrho(\mathfrak{A}(\mathbb{G}))''$. We suppose that the partition $\{D_n\}_n$ is infinite. Since $0\leq\sum_{n=1}^N\tilde{1}_{g_n(D_n)}F_{g_n}\nearrow F_\phi$ as $N\to\infty$ $P_l$-a.s., we obtain \[\left\|\varpi_l(F_\phi-\sum_{n=1}^N\tilde{1}_{g_n(D_n)}F_{g_n})\eta\right\|\to0\quad\text{as }N\to\infty\] for any $\eta\in L^2(\mathcal{R},\mathcal{B},P_l)$ by the dominated convergence theorem. Therefore, $\varpi_l(\sum_{n=1}^N\tilde{1}_{g_n(D_n)}F_{g_n})$ converges to $\varpi_l(F_\phi)$ in the strong operator topology, that is, $\varpi_l(F_\phi)$ also belongs to $\varrho(\mathfrak{A}(\mathbb{G}))''$. Hence we are done. 
\end{proof}

\medskip
Recall that for any $\rho\in\widehat{G}_N$ and any $t,u\in\Omega(*,\rho)$ every pair $(\omega,\omega')\in(C_t\times C_u)\cap\mathcal{R}_N$ can be written as $(\omega,\omega')=((t,\omega_{N+1},\omega_{N+2},\dots),(u,\omega_{N+1},\omega_{N+2},\dots))$. Using this and Formula \eqref{Eq:R6.1}, we have \begin{align*}\langle\varrho_N(S_tS_u^*)\eta_D,\eta_D\rangle&=\int_\Omega\sum_{\omega'\in[\omega]_N}[\varrho_N(S_tS_u^*)\eta_D](\omega,\omega')\overline{\eta_D(\omega,\omega')}\,dP(\omega)\\&=\int_\Omega f_{t,u}(\omega,\omega)\,dP(\omega)\\&=\delta_{t,u}P(C_t)\\&=\delta_{t,u}\frac{w(t)P(X_N=\rho)}{w\mathchar`-\dim(\rho)}=\chi(S_tS_u^*).\end{align*} Therefore, we obtain that $(L^2(\mathcal{R},\mathcal{B},P),\varrho,\eta_D)$ is the GNS-triple of $\mathfrak{A}(\mathbb{G})$ associated with $\chi$. Then we can prove Theorem \ref{Theorem:Ergodic_ExtremalKMS}.

\begin{proof}(Theorem \ref{Theorem:Ergodic_ExtremalKMS})
 By \cite[Theorem 5.3.30(3)]{BratteliRobinson2}, the von Neumann algebra $\varrho(\mathfrak{A}(\mathbb{G}))''=W^*(\mathcal{R};P)$ is factor if and only if the KMS state $\chi$ is extremal in the simplex $\mathrm{KMS}(\mathfrak{A}(\mathbb{G}),\widehat{\tau}^\mathbb{G})$. Therefore, by Proposition \ref{Prop:IntegralRepresentation}, the von Neumann algebra $W^*(\mathcal{R};P)$ is a factor if and only if the corresponding character $\chi$ is extremal in $\mathrm{Ch}(\mathbb{G})$. Theorem \ref{Theorem:Ergodic_ExtremalKMS} follows from this and the property of the Krieger construction mentioned in the first half of this subsection, that is, a von Neumann algebra constructed by the Krieger construction is a factor if and only if the corresponding probability measure is $\mathfrak{S}(\Omega)$-ergodic. 
\end{proof}

\subsubsection{The Ergodic method}
Theorem \ref{theorem:ergodicmethod} is a type of claim called the \emph{ergodic method}. Let $\mathcal{G}=(V,E,s,r)$ be a branching graph with a weight function $w\colon E\to(0,\infty)$. We define $X_N\colon\Omega\to V$ by $X_N(\omega)=r(\omega_N)$ for any $\omega=(\omega_N)_N\in\Omega$.

\begin{theorem}[the ergodic method]\label{theorem:ergodicmethod}
If $P$ is $\mathfrak{S}(\Omega)$-ergodic and $w$-central then there exists a path $\omega\in\Omega$ such that \[\frac{P(X_K=v)}{w\mathchar`-\dim(v)}=\lim_{N\to\infty,K\leq N}\frac{w\mathchar`-\dim(v,X_N(\omega))}{w\mathchar`-\dim(X_N(\omega))}\] for any $v\in V_K$ and $k\geq1$. In particular, if a weighted branching graph $\mathcal{G}$ associated with an inductive system of CQGs, then this holds true for any extremal $w$-central probability measures.
\end{theorem}

The first part of this theorem is obtained by the standard method using the backwards martingale convergence theorem and the ergodicity of probability measures. By Theorem \ref{Theorem:Ergodic_ExtremalKMS}, $\mathfrak{S}(\Omega)$-ergodic $w$-central probability measures coincide with extremal $w$-central probability measures. Thus, the second part follows.   

\begin{proof}
We fix the vertex $v\in V_K.$ For any $N\geq K+1,$ the $\sigma$-algebra $\mathcal{E}_N$ is generated by $X_N,X_{N+1},\dots$ and the random variables $Z_N$ on the probability space $(\Omega,\mathcal{F},P)$ are defined by \[Z_N(\omega):=\frac{w\mathchar`-\dim(v,X_N(\omega))}{w\mathchar`-\dim(X_N(\omega))}\] for $N=K+1,K+2,\dots$. Note that $\mathcal{E}_{K+1}\supset\mathcal{E}_{K+2}\supset\cdots$ and $|Z_N|\leq1/w\mathchar`-\dim(v)$ for any $N$. We claim that the stochastic process $(Z_N)_{N=K+1}^\infty$ is a $\{\mathcal{E}_N\}_N$-backward martingale. Clearly, $Z_N$ is $\mathcal{E}_N$-adapted and integrable. Thus it suffices to show that $\mathbb{E}[Z_N|\mathcal{E}_{N+1}]=Z_{N+1}$ $P$-almost surely. For any $L\geq N+1$ and any $u_n\in V_n$ with $n=N,\dots,L$ let $A:=\cap_{n=N+1}^L(X_n=u_n)\in\mathcal{E}_{N+1}$ and assume $A\neq\emptyset$. It suffices to show that $\mathbb{E}[Z_N1_A]=\mathbb{E}[Z_{N+1}1_A]$ where $1_A$ denotes the characteristic function of $A$. Indeed, for each path $t$ from $u_{N+1}$ to $u_L$ along $u_{N+2},\dots,u_{L-1}$, we have \[\mathbb{E}[Z_{N+1}1_A]=\frac{w\mathchar`-\dim(v,u_{N+1})w(t)}{w\mathchar`-\dim(u_L)}P(X_L=u_L)=\mathbb{E}[Z_N1_A].\] Therefore, by the backward martingale theorem, see e.g.\ \cite[Section 5.6]{Durrett}, $Z_N$ converges to $\mathbb{E}[Z_{K+1}|\mathcal{E}_\infty]$ $P$-almost surely as well as in $L^1$-norm, where $\mathcal{E}_\infty:=\cap_{N=K+1}^\infty\mathcal{E}_N$. Since the collection of all $\mathfrak{S}(\Omega)$-invariant sets forms a $\sigma$-algebra, every measurable set in $\mathcal{E}_\infty$ is $\mathfrak{S}(\Omega)$-invariant. Since $P$ is $\mathfrak{S}(\Omega)$-ergodic, we have \[\mathbb{E}[Z_{K+1}|\mathcal{E}_\infty]=\mathbb{E}[Z_{K+1}]=\frac{P(X_K=v)}{w\mathchar`-\dim(v)},\quad P\mathchar`-\mathrm{a.s.}\] Therefore, \[\frac{P(X_K=v)}{w\mathchar`-\dim(v)}=\lim_{N\to\infty,K\leq N}\frac{w\mathchar`-\dim(v,X_N(\omega))}{w\mathchar`-\dim(X_N(\omega))}\] for $P$-almost sure $\omega\in\Omega.$ Since the vertex set $V$ is countable, we are done. The final part of the theorem follows from Theorem \ref{Theorem:Ergodic_ExtremalKMS}.
\end{proof}

\medskip
When the weighted branching graph $(\mathcal{G},w)$ is associated with an inductive system $\mathbb{G}$ of CQGs, we have the following corollary.
\begin{corollary}\label{Cor6.2}
For any extremal quantized character $\chi\in\mathrm{ex}(\mathrm{Ch}(\mathbb{G}))$, there exists a sequence $\pi(1)\prec\pi(2)\prec\cdots$ such that \[\chi|_{C^*(G_K)}=\lim_{N\to\infty,K\leq N}\chi_{\pi(N)}\circ\Theta_{N,K}\] on $C^*(G_N)$, where $\Theta_{N,K}:=\Theta_{N-1}\circ\Theta_{N-2}\circ\cdots\circ\Theta_{K}$.
\end{corollary}
\begin{proof}
Firstly, we give the following simple observation: Let $P_K,\ P_K^i,\ i=1,2,\dots$ be probability measures on $\widehat{G}_K$. If $P_K^i$ converges to $P_K$ weakly, then we have $\sum_{\pi\in\widehat{G}_K}|P_K(\pi)-P_K^i(\pi)|\to0$ as $i\to\infty$. In order to show this, for any $\epsilon>0$ we take a finite subset $A\subset\widehat{G}_K$ such that $P_K(A)>1-\epsilon/3$. (This can be done since $\widehat{G}_K$ is at most countable.) Since $A$ is a finite set, there exists $i_0$ such that for any $i\geq i_0$ $P_K^{i}(A)>1-\epsilon/3$ and $|P_K(\{\pi\})-P_K^i(\{\pi\})|<\epsilon/3|A|$ for any $\pi\in A$. Then we have $\sum_{\mu\in\mathrm{Sign}_K}|P_K(\{\pi\})-P_K^i(\{\pi\})|<\epsilon$.

We show the corollary. Let $(P_K)_{K=1}^\infty$ be the $w$-coherent system corresponding to the extremal quantized character $\chi$. Remark that the corresponding $w$-central probability measure is $\mathfrak{S}(\Omega)$-ergodic by Theorem \ref{Theorem:Ergodic_ExtremalKMS}. Therefore, by the ergodic method (Theorem \ref{theorem:ergodicmethod}), there exists a path $\omega\in\Omega$ such that \[\frac{P_K(\pi)}{w\mathchar`-\dim(\pi)}=\lim_{N\to\infty,K\leq N}\frac{w\mathchar`-\dim(\pi,X_N(\omega))}{w\mathchar`-\dim(X_N(\omega))}\] for any $\pi\in\widehat{G}_N$. By Proposition \ref{P6.2}, Formula \eqref{Eq:Density} and the observation given in the first paragraph, for any $x=(x_\pi)_{\pi\in\widehat{G}_K}\in C^*(G_K)$ we have \begin{align*}\chi(x)&=\sum_{\pi\in\widehat{G}_K}P_K(\pi)\chi_\pi(x_\pi)\\&=\lim_{N\to\infty,K\leq N}\sum_{\pi\in\widehat{G}_K}w\mathchar`-\dim(\pi)\frac{w\mathchar`-\dim(\pi,X_N(\omega))}{w\mathchar`-\dim(X_N(\omega))}\chi_\pi(x_\pi)\\&=\lim_{N\to\infty,K\leq N}\sum_{\pi\in\widehat{G}_K}\chi_{X_N(\omega)}(\Theta_{N,K}(x)).\end{align*} Hence we are done.
\end{proof}

\section{The case of quantum unitary groups}
\subsection{Quantum unitary groups}
In this section, we review some basic facts of the quantum universal enveloping algebra $\mathcal{U}_q(\mathfrak{gl}(N))$ and the quantum unitary group $U_q(N)$. We suppose that $q$ belongs to the interval $(0,1)$ throughout the rest of the paper. 

\medskip
The quantum universal enveloping algebra $\mathcal{U}_q(\mathfrak{gl}(N))$ is a unital algebra generated by the letters $Q_i$, $Q_i^{-1}$, $E_j$, $ F_j,$ $i=1,\dots,N,j=1,\dots,N-1$ with the following relations:
\[Q_iQ_j=Q_jQ_i,\quad Q_iQ_i^{-1}=Q_i^{-1}Q_i=1,\]\[Q_iE_jQ_i^{-1}=q^{\delta_{i,j}/2-\delta_{i,j+1}/2}E_j,\quad Q_iF_jQ_i^{-1}=q^{-\delta_{i,j}/2+\delta_{i,j+1}/2}F_j,\]\[E_iF_j-F_jE_i=\delta_{i,j}\frac{Q_i^2Q_{i+1}^{-2}-Q_i^{-2}Q_{i+1}^2}{q-q^{-1}},\]\[ E_iE_j=E_jE_i,\quad F_iF_j=F_jF_i,\ |i-j|\geq2,\]\[E_j^2E_{j\pm1}-(q+q^{-1})E_jE_{j\pm1}E_j+E_{j\pm1}E_j^2=0,\]\[F_j^2F_{j\pm1}-(q+q^{-1})F_jF_{j\pm1}F_j+F_{j\pm1}F_j^2=0.\] 
Furthermore, the quantum unitary group $U_q(N)=(A_N,\delta_N)$ is a CQG whose unital $C^*$-algebra $A_N$ is generated by the letters $\mathrm{det}_q^{-1}(N)$ and $u_{ij}(N),$ $i,j=1,\dots,N$ with the following relations:
\begin{align*}\begin{aligned}u_{ij}(N)u_{kj}(N)&=qu_{kj}(N)u_{ij}(N),\quad i<k,\\u_{ij}(N)u_{il}(N)&=qu_{il}(N)u_{ij}(N),\quad j<l,\\u_{ij}(N)u_{kl}(N)&=u_{kl}(N)u_{ij}(N),\quad i<k,\ j>l,\\u_{ij}(N)u_{kl}(N)-qu_{il}(N)u_{kj}(N)&=u_{kl}(N)u_{ij}(N)-q^{-1}u_{kj}(N)u_{il}(N),\quad i<k,\ j<l\\x_{ij}(N)\mathrm{det}_q^{-1}(N)=\mathrm{det}_q^{-1}(N)x_{ij}(N),&\quad \mathrm{det}_q(N)\mathrm{det}_q^{-1}(N)=\mathrm{det}_q^{-1}(N)\mathrm{det}_q(N)=1,\end{aligned}\end{align*}
where $\mathrm{det}_q(N)$ is the so-called \emph{quantum determinant}. See for instance \cite{UenoTakebayashiShibukawa1}, \cite{NoumiYamadaMimachi} or \cite{KilSch}.

\medskip
It is known that every finite dimensional irreducible left $\mathcal{U}_q(\mathfrak{gl}(N))$-module $V$ has a highest weight vector and its weight becomes of the form $(\omega_1q^{\nu_1/2},\dots,\omega_Nq^{\nu_N/2})$, where $\omega_i\in\{\pm1,\pm\sqrt{-1}\}$ and $\nu=(\nu_1,\dots,\nu_N)\in\mathrm{Sign}_N$, see \cite[Proposition 1.2]{UenoTakebayashiShibukawa2} or \cite[Chapter 7]{KilSch}. Conversely, every weight of this form is a highest weight of some finite dimensional irreducible left $\mathcal{U}_q(\mathfrak{gl}(N))$-module and every finite dimensional left $\mathcal{U}_q(\mathfrak{gl}(N))$-module with highest weight vector must be irreducible. Let $\mathcal{H}_\nu$ be an irreducible left $\mathcal{U}_q(\mathfrak{gl}(N))$-module with the highest weight $(q^{\nu_1/2},q^{\nu_2/2},\dots,q^{\nu_N/2})$ for a signature $\nu\in\mathrm{Sign}_N$. Then we can obtain a concrete basis of finite dimensional irreducible left $\mathcal{U}_q(\mathfrak{gl}(N))$-module $\mathcal{H}_\nu$. In fact, there exists a basis $(v_t)_{t\in\Omega(*,\nu)}$ of $\mathcal{H}_\nu$ whose indices are the paths on the Gelfand--Tsetlin graph from $*$ to $\nu$, and each vectors $v_t$ is a weight vector with the weight \begin{equation}\label{weight}(q^{|r(t_1)|/2},q^{(|r(t_2)|-|r(t_1)|)/2},\dots,q^{(|r(t_N)|-|r(t_{N-1})|)/2}).\end{equation} This basis is called \emph{the Gelfand--Tsetlin basis}. See \cite{UenoTakebayashiShibukawa1}, \cite{UenoTakebayashiShibukawa2}, \cite[Section 7.3]{KilSch} for more details.

\medskip
Let $\mathcal{A}_N$ be the dense $*$-subalgebra generated by the letters $\mathrm{det}_q^{-1}(N)$ and $u_{ij}(N)$ with $i,j=1,\dots,N$, and hence the $C^*$-algebra $A_N$ is the universal enveloping unital $C^*$-algebra generated by $\mathcal{A}_N$. Then this algebra $\mathcal{A}_N$ has a Hopf $*$-algebra structure with the coproduct $\delta_N|_{\mathcal{A}_N}$. We denote this Hopf $*$-algebra by $\mathcal{U}_q(N)$ and call it the \emph{algebraic quantum unitary group}. It is known that there exists a dual pairing between $\mathcal{U}_q(\mathfrak{gl}(N))$ and $\mathcal{U}_q(N)$, that is, \[(\,\cdot\,,\,\cdot\,)\colon\mathcal{U}_q(\mathfrak{gl}(N))\times\mathcal{A}_N\to\mathbb{C},\] see \cite[Proposition 1.3]{NoumiYamadaMimachi}. This pairing naturally induces the representation of the algebra $\mathcal{U}_q(\mathfrak{gl}(N))$ from a right coaction of $\mathcal{U}_q(N)$. Indeed, for any right coaction $\pi\colon V\to V\otimes\mathcal{A}_N$ we obtain the representation $\hat{\pi}\colon\mathcal{U}_q(\mathfrak{gl}(N))\to\mathcal{B}(V)$ defined by $\hat{\pi}(x):=(\mathrm{id}\otimes f_x)\pi$, where $f_x\colon\mathcal{A}_N\to\mathbb{C}$ is defined by $f_x(\,\cdot\,)=(x,\,\cdot\,)$. Thus, (highest) weights and (highest) weight vectors of right $\mathcal{U}_q(N)$-comodules make sense. One of the fundamental facts in the representation theory of $U_q(N)$ is that the (equivalence classes of) irreducible unitary representations are parametrized by signatures in $\mathrm{Sign}_N$, and the irreducible representation $U_\nu$ corresponding to $\nu\in\mathrm{Sign}_N$ has the highest weight $(q^{\nu_1/2},q^{\nu_2/2},\dots,q^{\nu_N/2})$. See \cite[Theorem 2.12, Theorem 3.7]{NoumiYamadaMimachi}. Therefore, using the weights in Equation \eqref{weight} as well as \cite[Theorem 3.7]{NoumiYamadaMimachi}, we can explicitly compute the density matrix $F_\nu$ as \begin{equation}\label{density}[F_\nu]_{tu}=\delta_{t,u}q^{-(N-1)|\nu|+2\sum_{n=1}^{N-1}|r(t_n)|},\end{equation} where $t,u\in\Omega(*,\nu)$ and $t=(t_n)_{n=1}^N$.

\medskip
The unitary quantum group $U_q(N)$ can be regarded as a quantum subgroup of $U_q(N+1)$ with the surjective unital $*$-homomorphism $\theta_N\colon A_{N+1}\to A_N$ defined by \begin{align*}&\theta_N(u_{ij}(N+1)):=\begin{cases}u_{ij}(N) & (1\leq i,j\leq N)\\\delta_{i,j}1&(\text{otherwise}),\end{cases}\\&\theta_N(\mathrm{det}_q^{-1}(N+1)):=\mathrm{det}_q^{-1}(N).\end{align*} Thus, the inductive system of $U_q(N)$ is well defined and denoted by $\mathbb{U}_q$. It is knwon that this inductive system $\mathbb{U}_q$ and the inductive system of $U(N)$ have the same branching rule, that is, the restriction of the irreducible representation corresponding to $\nu=(\nu_1,\dots,\nu_{N+1})\in\mathrm{Sign}_{N+1}$ contains the irreducible representation corresponding to $\mu=(\mu_1,\dots,\mu_N)\in\mathrm{Sign}_N$ if and only if \[\nu_1\geq\mu_1\geq\nu_2\geq\cdots\geq\nu_N\geq\mu_N\geq\nu_{N+1},\] and this rule is independent of the $q$-deformation. See \cite[Section 4]{NoumiYamadaMimachi}. Therefore, the branching graph associated with the inductive system $\mathbb{U}_q$ also becomes the Gelfand--Tsetlin graph $\mathbb{GT}$.

\subsection{Main theorem}
By Equation \eqref{density}, the weighted branching graph associated with $\mathbb{U}_q$ consists of the Gelfand--Tsetlin graph $\mathbb{GT}$ and the weight function $w$ defined by \begin{equation}\label{Def:weight}w((\mu,\nu))=q^{N|\mu|-(N-1)|\nu|},\end{equation} for $\mu\in\mathrm{Sign}_{N-1}$ and $\nu\in\mathrm{Sign}_N$. Recall that the rational Schur polynomial $s_\nu(x_1,\dots,x_N)$ suffixed by a signature $\nu\in\mathrm{Sign}_N$ is defined as \[s_\nu(x_1,\dots,x_N):=\frac{\det\left[x_i^{\nu_j+N-j}\right]_{i,j=1}^N}{\prod_{1\leq i<j\leq N}(x_i-x_j)},\] and satisfies the following branching rule: \[s_\nu(x_1,\dots,x_N)=\sum_{\mu\in\mathrm{Sign}_{N-1}:\mu\prec\nu}x_N^{|\nu|-|\mu|}s_\mu(x_1,\dots,x_{N-1}).\] Thus, we have \[w\mathchar`-\dim(\nu)=\sum_{t\in\Omega(*,\nu)}\prod_{n=1}^Nq^{n|s(t_n)|-(n-1)|r(t_n)|}=s_\nu(q^{N-1},q^{N-3},\dots,q^{-N+1})=\dim_q(\nu).\] See \cite[Section 3]{NoumiYamadaMimachi} for the latter equation. On the other hand, a probability measure $P$ on the paths on the Gelfand--Tsetlin graph $\mathbb{GT}$ is $w$-central if and only if for any signature $\nu$ and any finite path $t$ from $*$ to $\nu$ the probability measure $P$ satisfies that\[P(C_t)=\frac{w(t)}{\dim_q(\nu)}P(X_N=\nu)=\frac{q^{2(|r(t_1)|+\cdots+|r(t_{N-1})|)}}{\sum_{s\in\Omega(*,\nu)}q^{2(|r(s_1)|+\cdots+|r(s_{N-1})|)}}P(X_N=\nu),\] where $X_N$ is the random variable on $(\Omega,P)$ defined by $X_N(\omega)=r(\omega_N)$ for any $\omega=(\omega_n)_{n=1}^\infty\in\Omega$. Therefore, we arrive at the main theorem.

\begin{theorem}\label{main}
There exists an affine homeomorphism between the simplex of quantized characters of the inductive system $\mathbb{U}_q$ and the simplex of $q^2$-central probability measures on the paths on the Gelfand--Tsetlin graph $\mathbb{GT}$. Moreover, this affine homeomorphism is determined by \[\chi|_{C^*(U_q(N))}=\sum_{\nu\in\mathrm{Sign}_N}P(X_n=\nu)\chi_\nu,\] where $\chi_\nu$ is the natural extension of $\mathrm{Tr}(F_\nu\,\cdot\,)/\dim_q(\nu)$ to $C^*(U_q(N))$.
\end{theorem}

\subsection{Quantized characters of the inductive system $\mathbb{U}_q$ of quantum unitary groups}
In this final section, we propose a representation-theoretic interpretation of Gorin's generating functions of probability measures on the sets $\mathrm{Sign}_N$ of signatures. 

Let $P_N$ be a probability measure on $\mathrm{Sign}_N$. Then the generating function $\mathcal{S}(x_1,\dots,x_N;P_N)$ of $P_N$ is defined to be \[\mathcal{S}(x_1,\dots,x_N;P_N):=\sum_{\nu\in\mathrm{Sign}_N}P_N(\{\nu\})\frac{s_\nu(x_1,\dots,x_N)}{s_\nu(1,q^{-2},\dots,q^{-2(N-1)})}.\] It is known that this generating function $\mathcal{S}$ converges uniformly on \[T_N:=\{(x_1,\dots,x_N) : |x_i|=q^{-2(i-1)},i=1,\dots,N\}\] and captures several important properties of $P_N$:
\begin{itemize}
\item A sequence of probability measures $P_N$ on $\mathrm{Sign}_N$ for each $N=1,2,\dots$ becomes a $q$-coherent system if and only if \[\mathcal{S}(x_1,\dots,x_N,q^{-2N};P_{N+1})=\mathcal{S}(x_1,\dots,x_N;P_N)\] for every $N=1,2,\dots$, see \cite[Proposition 4.6]{Gorin:qcentralmeasure}. 
\item A sequence of probability measures $P_N^i$ on $\mathrm{Sign}_N$ converges to a probability measure $P_N$ weakly if and only if the functions $\mathcal{S}(x_1,\dots,x_N;P_N^i)$ uniformly converge to the function $\mathcal{S}(x_1,\dots,x_N;P_N)$ on $T_N$, see \cite[Proposition 4.10, Proposition 4.11]{Gorin:qcentralmeasure}. 
\end{itemize}
Remark that the generating functions $\mathcal{S}$ are main tool in Gorin's analysis of $q$-central probability measures, see \cite{Gorin:qcentralmeasure}.

\medskip
In what follows, we will observe that the restriction of an arbitrary quantized character of $\mathbb{U}_q$ to the maximal torus of each $U_q(N)$ is exactly one of Gorin's generating functions. We think that this observation explains why Gorin's generating functions plays a key role in Gorin's work on $q$-central probability measures.

\medskip
We firstly introduce the quantum group $T^N:=(C(\mathbb{T}^N),\delta_{T^N})$ of the compact group $\mathbb{T}^N$, see \cite[Example 1.1.2]{Neshveyev}. The continuous functions $t_1,\dots,t_N,t_1^{-1},\dots,t_N^{-1}\colon\mathbb{T}^N\to\mathbb{C}$ are defined to be \[t_i(z_1,\dots,z_N):=z_i,\quad t_i^{-1}(z_1,\dots,z_N):=\overline{z_i}\] for any $i=1,\dots,N$. Then the functions $t_1,\dots,t_N,t_1^{-1},\dots,t_N^{-1}$ generate $C(\mathbb{T}^N)$ and we have $\delta_{T^N}(t_i)=t_i\otimes t_i$. Remark that $T^N$ is a quantum subgroup of the quantum unitary group $G_N$ by the surjective $*$-homomorphism $\pi_{T^N}\colon A_N\to C(\mathbb{T}^N)$ defined by \[\pi_{T^N}(u_{ij}(N)):=\delta_{i,j}t_i,\quad \pi_{T^N}(\mathrm{det}_q^{-1}(N)):=t_1^{-1}\cdots t_N^{-1}.\] Furthermore, $T^N$ is a quantum subgroup of the quantum group $T^{N+1}$ by the surjective $*$-homomorphism $\theta_{T^N}\colon C(\mathbb{T}^{N+1})\to C(\mathbb{T}^N)$ defined by \[[\theta_{T^N}f](z_1,\dots,z_N):=f(z_1,\dots,z_N,1)\] for any $f\in C(\mathbb{T}^{N+1})$. Remark that $\theta_{T^N}\circ\pi_{T^{N+1}}=\pi_{T^N}\circ\theta_N$ for any $N\geq1$. 

\medskip
Let \[V_N:=(U_\nu)_{\nu\in\widehat{G}_N}\in \bigoplus_{\nu\in\widehat{G}_N}B(\mathcal{H}_\nu)\otimes A_N\cong W^*(G_N)\otimes A_N.\] Then we have the following theorem.
\begin{theorem}\label{Theorem:rep_qcharacters}
For any $\nu\in\mathrm{Sign}_N$ and the state $\chi_\nu$ on $W^*(U_q(N))$ defined to be $\mathrm{Tr}(F_\nu p_\nu\ \cdot)/\dim_q(\nu)$ (see Section \ref{SVAF}) we have \begin{equation}\label{Eq:10.1}\Phi_{\chi_\nu}^{(N)}(z_1,\dots,z_N):=(\chi_\nu\otimes\pi_{T^N})(V_N)(z_1,\dots,z_N)=\frac{s_\nu(z_1,q^{-2}z_2,\dots,q^{-2(N-1)}z_N)}{s_\nu(1,q^{-2},\dots,q^{-2(N-1)})}.\end{equation} Moreover, for any probability measure $P_N$ on $\mathrm{Sign}_N$ and the character $\chi:=\sum_{\nu\in\widehat{G}_N}P_N(\nu)\chi_\nu$ we also have \begin{equation}\label{Eq:10.2}\Phi_\chi^{(N)}(z_1,\dots,z_N):=(\chi\otimes\pi_{T^N})(V_N)(z_1,\dots,z_N)=\sum_{\nu\in\mathrm{Sign}_N}P_N(\nu)\Phi_\nu(z_1,\dots,z_N).\end{equation} Therefore, we obtain that \[\Phi^{(N)}_\chi(x_1,\dots,x_N)=\mathcal{S}(z_1,q^{-2}z_2\dots,q^{-2(N-1)}z_N;P_N).\] 
\end{theorem}
\begin{proof}
Equation \eqref{Eq:10.1} follows from \cite[Section 3.2]{NoumiYamadaMimachi} and Equation \eqref{Eq:10.1} is immediate from Equation \eqref{Eq:10.1}.
\end{proof}

The following corollary is immediate from properties of the generating function $\mathcal{S}$.
\begin{corollary}
For any quantized character $\chi,\chi_n\in\mathrm{Ch}(\mathbb{U}_q),n=1,2,\dots$
\begin{itemize}
\item $\Phi_\chi^{(N+1)}(z_1,\dots,z_N,1)=\Phi_\chi^{(N)}(z_1,\dots,z_N)$ for any $N\geq1$, 
\item if $\chi_n\to\chi$ as $n\to\infty$ in the weak$^*$ topology, then the functions $\Phi_{\chi_n}^{(N)}$ converge to $\Phi_\chi^{(N)}$ uniformly on $\mathbb{T}^N$ as $n\to\infty$ for any $N\geq1$,
\end{itemize}
\end{corollary}
By the former property, we obtain the function $\Phi_\chi$ on \[T:=\{(z_1,\dots,z_N,1,\dots):z_i\in\mathbb{T},i=1,\dots,N,N\geq1\}\cong\lim_\to\mathbb{T}^N\] defined by $\Phi_\chi(z_1,\dots,z_N,1,\dots):=\Phi_\chi^{(N)}(z_1,\dots,z_N)$. Furthermore, by these properties and the ergodic method (Corollary \ref{Cor6.2}), we have the following proposition:

\begin{proposition}
For any extremal quantized character $\chi\in\mathrm{ex}(\mathrm{Ch}(\mathbb{U}_q))$, there exists a sequence $\nu(1)\prec\nu(2)\prec\cdots$ of signatures such that \begin{align*}\Phi_\chi(z_1,\dots,z_N,1,\dots)&=\lim_{L\to\infty,N\leq L}\Phi_{\chi_{\nu(L)}}^{(L)}(z_1,\dots,z_N,1\dots,1)\\&=\lim_{L\to\infty,N\leq L}\frac{s_{\nu(L)}(z_1,q^{-2}z_2,\dots,q^{-2(N-1)}z_N,q^{-2N},\dots,q^{-2(L-1)})}{s_{\nu(L)}(1,q^{-2},\dots,q^{-2(L-1)})}\end{align*} for any $(z_1,\dots,z_N,1,\dots)\in T$ and any $N\geq1$.
\end{proposition}

Remark that we can describe an explicit formula of $\Phi_\chi$ for any extremal quantized character of the inductive system $\mathbb{U}_q$ by \cite[Theorem 6.5]{GorinPanova}. In this section we have considered, in a more direct fashion, the generating functions $\mathcal{S}$ of probability measures on the set of signatures defined by Gorin in \cite{Gorin:qcentralmeasure} in relation to quantized characters of the inductive system $\mathbb{U}_q$. We hope that the consideration here is also useful for the representation theory of the $\sigma$-$C^*$-quantum group $U_q(\infty)$ introduced in \cite{MahantaMathai}.

\appendix
\section{Boundary theorem}
In this appendix we briefly explain of the boundary theorem of the $q$-Gelfand--Tsetlin graph below as an application of the ergodic method. Let $\mathcal{N}:=\{\theta=(\theta_i)_{i=1}^\infty\in\prod_{i=1}^\infty\mathbb{Z}:\theta_1\leq\theta_2\leq\cdots\}$ endowed with the topology of component-wise convergence. 

\begin{theorem}\label{BoundaryTheorem}
There exists a homeomorphism from the set of extremal quantized characters $\mathrm{ex}(\mathrm{Ch}(\mathbb{U}_q))$ of the inductive system $\mathbb{U}_q$ to the set $\mathcal{N}$.
\end{theorem}

The correspondence from $\mathrm{ex}(\mathrm{Ch}(\mathbb{U}_q))$ to $\mathcal{N}$ is immediately obtained by the ergodic method. Indeed, by the ergodic method (Corollary \ref{Cor6.2}), for any extremal quantized character $\chi\in\mathrm{ex}(\mathrm{Ch}(\mathbb{U}_q))$ there exists a sequence $\nu^{(1)}\prec\nu^{(2)}\prec\cdots$ such that $\chi|_{C^*(U_q(K))}=\lim_{N\to\infty,K\leq N}\chi_{\nu^{(N)}}\circ\Theta_{N,K}$. Then, using \cite[Theorem 1.3.1 (or Theorem 5.1)]{Gorin:qcentralmeasure}, we can conclude that this sequence $(\nu^{(N)})_{N=1}^\infty$ is stable, that is, the number $\nu^{(N)}_{N-i+1}$ converges as $N\to\infty$ for any $i\geq1$. Therefore, we obtain the corresponding parameter $(\theta_i)_{i=1}^\infty$ by $\theta_i:=\lim_{N\to\infty}\nu^{(N)}_{N-i+1}$. Remark that this parameter does not depend on the sequence $(\nu^{(N)})_{N=1}^\infty$ that we have chosen and depends only on the given extremal quantized character $\chi$. See \cite{Gorin:qcentralmeasure} for more details as well as the converse correspondence.

\begin{remark}
A finer classification of extremal quantized characters of the inductive system $\mathbb{U}_q$ seems possible by using the Murray--von Neumann--Connes type classification for von Neumann factors. In fact, we have obtained the following: Let $\chi$ be the extremal quantized character corresponding $\theta=(\theta_i)_{i=1}^\infty\in\mathcal{N}$. Then the GNS representation $\varrho_\chi$ of the Stratila--Voiculescu AF-algebra $\mathfrak{A}(\mathbb{U}_q)$ for $\chi$ is type I, that is, the von Neumann factor $\varrho_\chi(\mathfrak{A}(\mathbb{U}_q))''$ is type I, if and only if the parameter $\theta=(\theta_i)_{i=1}^\infty$ is bounded. The details on this classification of other cases will be discussed elsewhere.
\end{remark}

\section{$(q,t)$-Central probability measures}
Cuenca \cite{Cuenca} introduced the concept of $(q,t)$-central probability measures on the paths on the Gelfand--Tsetlin graph $\mathbb{GT}$. However, we do not know at the present moment whether or not these $(q,t)$-central probability measures come from some quantum group deformation of $U(N)$. This is indeed an interesting question. Nevertheless, the $(q,t)$-central probability measures can still be dealt with in our framework without appealing quantum groups. Namely, in this appendix, we will introduce the corresponding $(q,t)$-analogue $F_\nu(q,t)$ of the density matrices $F_\nu$ in Equation \eqref{density} (by slightly changing parameters for convenience), and this $(q,t)$-analogue gives a new Stratila--Voiculescu AF-flow successfully on the same Stratila--Voiculescu AF-algebra as that of $\mathbb{U}_q$.

Cuenca's idea to introduce the concept of $(q,t)$-central probability measures is replacing the quantum dimension $\dim_q(\nu)=S_\nu(q^{N-1},q^{N-3},\dots,q^{-N+1})$ of the signature $\nu\in\mathrm{Sign}_N$ with \[P_\nu(t^{N-1},t^{N-3},\dots,t^{-N+1};q,t^2),\] where $P_\nu$ is the (rational) Macdonald polynomial assigned to $\nu$, and $q,t\in (0,1)$ are the Macdonald parameters. Using the branching rule for the Macdonald polynomials (see \cite{Macdonald:book} or \cite[Theorem 2.5]{Cuenca}), we can interpret these values $P_\nu(t^{N-1},t^{N-3},\dots,t^{-N+1};q,t^2)$ as the weighted dimensions of the vertices $\nu$ of the Gelfand--Tsetlin graph with the weight function $w_{q,t}$ defined by \[w_{q,t}((\mu,\nu)):=\psi_{\nu/\mu}(q,t^2)t^{N|\mu|-(N-1)|\nu|},\quad \mathrm{Sign}_{N-1}\ni\mu\prec\nu\in\mathrm{Sign}_N,\] where $\psi_{\nu/\mu}(q,t^2)$ is the branching coefficient (see \cite{Macdonald:book} or \cite[Theorem 2.5]{Cuenca} for its explicit form). Moreover, any $w_{q,t}$-central probability measure $M$ satisfies that for any finite path $u$ from $*$ to $\nu\in\mathrm{Sign}_N$ \begin{align*}M(C_u)&=\frac{w_{q,t}(u)}{P(t^{N-1},t^{N-3},\dots,t^{-N+1};q,t^2)}M(X_N=\nu)\\&=t^{2(|r(u_1)|+\cdots+|r(u_{N-1})|)}\frac{\psi_{r(u_1)/r(u_0)}(q,t^2)\cdots\psi_{r(u_N)/r(u_{N-1})}(q,t^2)}{P_\nu(1,t^2,\dots,t^{2(N-1)};q,t^2)}M(X_N=\nu),\end{align*} where the $\mathrm{Sign}_N$-valued random variable $X_N$ is defined by $X_N(\omega)=r(\omega_N)$ for any $\omega\in\Omega$. This is nothing but a $(q,t^2)$-central probability measure (see \cite[Definition 6.4]{Cuenca}). 

With these preparations, by Formula \eqref{Eq:R6.1}, we define the $(q,t)$-analogue $F_\nu(q,t)$ to be \[F_\nu(q,t)=\sum_{u\in\Omega(*,\nu)}w_{q,t}(u)S_uS_u^*,\]that is, its matrix representation becomes \[[F_\nu(q,t)]_{uv}=\delta_{u,v}\prod_{k=1}^N\psi_{r(u_k)/s(u_k)}(q,t^2)t^{K|\mu|-(K-1)|\nu|},\quad u,v\in\Omega(*,\nu).\] Using the branching rule for Macdonald polynomials, we can confirm that \[\mathrm{Tr}(F_\nu(q,t))=P_\nu(t^{N-1},t^{N-3},\dots,t^{-N+1};q,t^2).\] Moreover, we can construct the Stratila--Voiculescu AF-flow form the $F_\nu(q,t)$ successfully on the same Stratila--Voiculescu AF-algebra as in the case of $\mathbb{U}_q$. Then we can easily confirm that Theorem \ref{Theorem:Ergodic_ExtremalKMS} and Theorem \ref{theorem:ergodicmethod} still hold in this case too. In particular, the extremity for $(q,t^2)$-central probability measures coincides with the ergodicity for $(q,t^2)$-central probability measures as in Section 2.5. It may be important to notice that the consideration here suggests that any Stratila--Voiculescu AF-flow on the same Stratila--Voiculescu AF-algebra as that of $\mathbb{U}_q$ (or equivalently that $U(\infty)$) can be interpreted as a possible deformation of the Gelfand--Tsetlin graph in the algebraic level. On the other hand, according to the consideration here, we can regard a generating function of extremal $(q,t^2)$-central probability measure (see Theorem 5.5 (or Section 6) in \cite{Cuenca}) as \emph{$(q,t)$-analogue of Voiculescu function}.

\section*{Acknowledgment}
The author gratefully acknowledges the passionate guidance and continuous encouragement from his supervisor, Professor Yoshimichi Ueda. The author thanks Professor Yuki Arano for his useful comments on quantum subgroups and Section 3.3 of this paper, and for letting the author know the work \cite{MeyerRoyWoronowicz}. The author also thanks Professor Reiji Tomatsu and Professor Makoto Yamashita for their comments on this paper. In particular, Professor Reiji Tomatsu had the author be aware of an inaccuracy in the author's understanding of Theorem 1.3.1 in \cite{Gorin:qcentralmeasure} and informed the author of the reference \cite{Tomatsu}. Finally, the author appreciates the referee for careful reading and useful comments. 
}

\end{document}